\documentclass{amsart}
\usepackage{a4wide,amsmath,amsfonts,amssymb,amsthm,color,mathtools}
\usepackage[all]{xy}

\newtheorem{theorem}{Theorem}[section]

\newtheorem{corollary}[theorem]{Corollary}
\newtheorem{lemma}[theorem]{Lemma}
\newtheorem{claim}[theorem]{Claim}

\theoremstyle{definition}
\newtheorem{definition}[theorem]{Definition}
\newtheorem{remark}[theorem]{Remark}

\newcommand{\e}{\varepsilon}
\newcommand{\Dim}{\mathrm{Dim}}
\newcommand{\SDim}{\mathsf S\mbox{-}\mathsf{Dim}}
\newcommand{\Sdim}{\mathsf S\mbox{-}\mathrm{dim}}
\newcommand{\Hdim}{\mathsf{H\ddot o}\mbox{-}\mathrm{dim}}
\newcommand{\HoDim}{\mathsf{H\ddot o}\mbox{-}\mathrm{Dim}}
\newcommand{\HDim}{\HoDim}
\newcommand{\Ra}{\Rightarrow}

\newcommand{\C}{\mathcal C}
\newcommand{\U}{\mathcal U}
\newcommand{\A}{\mathcal A}
\newcommand{\M}{\mathcal M}
\newcommand{\w}{\omega}

\newcommand{\IR}{\mathbb R}

\newcommand{\Lip}{\mathrm{Lip}}
\newcommand{\defeq}{\coloneqq}%{\overset{\mathsf{def}}=}

\newcommand{\diam}{\mathrm{diam}}
\newcommand{\F}{\mathcal F}

\newcommand{\St}{\mathcal F^\star}
\newcommand{\IN}{\mathbb N}
\newcommand{\mesh}{\mathrm{mesh}}

\newcommand{\TM}[1]{\textcolor{blue}{#1}}

\usepackage[normalem]{ulem}

\title[A controlled Hahn-Mazurkiewicz Theorem and its applications]{A controlled Hahn-Mazurkiewicz Theorem\\ and its applications}
\author{Taras Banakh, Tetiana Martyniuk, Magdalena Nowak, Filip Strobin}
\address{T. Banakh: Ivan Franko National University of Lviv (Ukraine), and Jan Kochanowski University in Kielce (Poland); ORCID 0000-0001-6710-4611}
\email{t.o.banakh@gmail.com}
\address{T.~Martyniuk: INRIA de Paris (France), and Ukrainian Catholic University (Ukraine)}
\email{tetyanka.martynyuk@gmail.com}
\address{M.~Nowak: Jan Kochanowski University in Kielce (Poland); ORCID 0000-0003-1915-0001}
\email{magdalena.nowak805@gmail.com}
\address{F. Strobin: Institute of Mathematics, Lodz University of Technology, Al. Politechniki 8, 93-590 Łódź (Poland);
ORCID 0000-0002-8671-9053}
\email{filip.strobin@p.lodz.pl}

\begin{document}

\begin{abstract} For a metric Peano continuum $X$, let $S_X$ be a Sierpi\'nski function assigning to each $\e>0$ the smallest cardinality of a cover of $X$ by connected subsets of diameter $\le \e$. We prove that for any increasing function $\Omega:\IR_+\to\IR_+$ with $(0,1]\subseteq\Omega[\IR_+]$ and $s:=\sum_{n=1}^\infty S_X(2^{-n})\sum_{m=n}^\infty S_X(2^{-m})\,\Omega^{-1}(\min\{1,2^{6-m}\})<\infty$ there exists a continuous surjective function $f:[0,s]\to X$ with continuity modulus $\w_f\le\Omega$. This controlled version of the classical Hahn-Mazurkiewicz Theorem implies that $\SDim(X)\le\HDim(X)\le 2{\cdot}\SDim(X)$, where $\SDim(X)=\limsup_{\e\to 0}\frac{\ln(S_X(\e))}{\ln(1/\e)}$ is the {\em $S$-dimension} of $X$\TM{,} and  $\HDim(X)=\inf\{\alpha\in (0,\infty]:$ there is a~surjective $\frac1\alpha$-H\"older map $f:[0,1]\to X\}$ is the {\em H\"older dimension} of $X$.
\end{abstract}

\maketitle

\section{Introduction}

A Hausdorff topological space $X$ is called a {\em Peano continuum} if there exists a surjective continuous map $f:[0,1]\to X$. 
By the classical Hahn-Mazurkiewicz Theorem \cite{H}, \cite{M} (see also \cite[8.14]{N}), a  Hausdorff topological space $X$ is a Peano continuum if and only if $X$ is compact, metrizable, connected, locally connected, and nonempty. 
According to Sierpi\'nski \cite{S} (see also \cite[8.4]{N}), a connected compact metric space $X$ is a Peano continuum if and only if it {\em has property S},  i.e., for every $\e>0$ the space $X$ has a finite cover by connected subsets of diameter $\le\e$. 
The smallest possible  cardinality of such a cover will be denoted by $S_X(\e)$. 
The Sierpi\'nski function $S_X$ can be used to analyze a decompositional complexity of a metric Peano continuum $X$.
 
Since every Peano continuum $X$ is the image of the interval $[0,1]$ under a continuous surjective map $f:[0,1]\to X$, it is natural to ask whether such a map $f$ can have some additional properties, for example, be Lipschitz or H\"older, or have some other restrictions on its continuity modulus.
 
Let $\IR_+:=(0,\infty)$ be the open half-line and $[0,\infty]:=[0,\infty)\cup\{\infty\}$ be the  closed half-line with added infinity. For a function $f:X\to Y$ between  metric spaces $(X,d_X)$ and $(Y,d_Y)$, the {\em continuity modulus} of $f$ is the function
$$\w_f:\IR_+\to[0,\infty],\quad\w_f:\delta\mapsto\sup(\{d_Y(f(x),f(x')):x,x'\in X,\;d_X(x,x')\le\delta\}\cup\{0\}).$$

A function $f:X\to Y$ between metric spaces is called {\em $\alpha$-H\"older} with {\em H\"older constant} $\alpha\ge 0$ if $\sup_{t\in\IR_+}\frac{\w_f(t)}{t^\alpha}<\infty$, which is equivalent to the existence of a positive real constant $C$ such that
$$d_Y(f(x),f(x'))\leq C\cdot d_X(x,x')^\alpha\quad\mbox{for every  $x,x'\in X$.}$$ $1$-H\"older functions are called {\em Lipschitz}. Observe that every function $f:X\to Y$ to a bounded metric space $Y$ is $0$-H\"older. 

The following quantitative generalization of the Hahn-Mazurkiewicz Theorem is the main result of this paper.

\begin{theorem}\label{t:main} Let $X$ be a nonempty compact connected metric space of diameter $\le 1$ and $\Omega:\IR_+\to\IR_+$ be an increasing function such that $(0,1]\subseteq \Omega[\IR_+]$. If $$s\defeq\sum_{n=1}^\infty S_X(2^{-n})\sum_{m=n}^\infty S_X(2^{-m}){\cdot}\Omega^{-1}(\min\{1,2^{6-m}\})<\infty,$$ then there exists a continuous surjective function $f:[0,s]\to X$ with continuity modulus $\w_f\le\Omega$.
\end{theorem} 

 The proof of Theorem~\ref{t:main} will be presented in Section~\ref{s:main} after long preliminary work made in Sections~\ref{s:covers}--\ref{s:HM}.
 
 Now we discuss an application of Theorem~\ref{t:main} to fractal dimensions of metric Peano continua. One of such dimensions in the {\em $S$-dimension} 
 $$\SDim(X)\defeq\limsup_{\e\to0}\frac{\ln S_X(\e)}{\ln(1/\e)},$$
 which is a ``connected'' counterpart of the classical {\em box-counting dimension}
 $$\Dim(X)\defeq\limsup_{\e\to0}\frac{\ln N_X(\e)}{\ln(1/\e)}$$where $N_X(\e)$ is the smallest cardinality of a cover of $X$ by subsets of diameter $\le \e$. 
 
Observe that for every real number $r$ we have the implications
 $$\SDim(X)<r\;\;\Ra\;\;\limsup_{\e\to0}\e^rS_X(\e)<\infty\;\;\Ra\;\;\SDim(X)\le r.$$ 

By \cite{BT}, the $S$-dimension $\SDim(X)$ of any metric Peano continuum $X$ is bounded from above by its {\em H\"older dimension}
$$\HoDim(X)\defeq\inf\{\alpha\in (0,\infty]:\mbox{there is a surjective $\tfrac1\alpha$-H\"older function $f:[0,1]\to X$}\}.$$
More precisely, by \cite[Theorem~3.1]{BT}, for every metric Peano continuum $X$ we have the inequalities
$$\Dim(X)\le\SDim(X)\le\HoDim(X).$$
By \cite[Corollary 7.3]{BT}, the box-counting dimension $\Dim(X)$ of a metric Peano continuum can be strictly smaller than its $S$-dimension $\SDim(X)$. On the other hand, it is not known whether $\SDim(X)=\HoDim(X)$ for every metric Peano continuum $X$, see \cite[Problem 4.5]{BT}. Nonetheless, the following corollary of Theorem~\ref{t:main} shows that the gap between $\SDim(X)$ and $\HDim(X)$ is not too large.

\begin{corollary}\label{c:main} Each metric Peano continuum $X$ has $$\SDim(X)\le\HoDim(X)\le 2\cdot \SDim(X).$$
\end{corollary}

\begin{proof} Multiplying the metric of $X$ by a suitable constant, we lose no generality assuming that $\diam(X)\le 1$. The inequality $\SDim(X)\le\HDim(X)$ was proved in \cite[Theorem 3.1]{BT}. The inequality $\HDim(X)\le 2{\cdot}\SDim(X)$ will follow as soon as we show that for every $\alpha>2{\cdot}\SDim(X)$ there exists a surjective $\frac1{\alpha}$-H\"older map $g:[0,1]\to X$. Choose any real number $r$ with $\SDim(X)<r<\alpha/2$. Since $\SDim(X)<r$, there exists a constant $C$ such that $S_X(\e)\le \frac{C}{\e^r}$ for all $\e\in(0,1]$. Consider the function $\Omega:\IR_+\to\IR_+$, $\Omega:t\mapsto t^{1/\alpha}$, and observe that
$$
\begin{aligned}
s&\defeq\sum_{n=1}^\infty S_X(2^{-n})\sum_{m=n}^\infty S_X(2^{-m})\Omega^{-1}(\min\{1,2^{6-m}\})\le \sum_{n=1}^\infty C2^{nr}\sum_{m=n}^\infty C2^{mr}2^{(6-m)\alpha}\\%C^252^{d'}\sum_{n=1}2^{dn}\sum_{m=n}2^{(d-d')m}=\\
%&=\frac{52^{d'}C^2}{1-2^{d-d'}}\sum_{n=1}^\infty 2^{(2d-d')n}=
&=\frac{2^{6\alpha}C^22^{2r-\alpha}}{(1-2^{r-\alpha})(1-2^{2r-\alpha})}<\infty.
\end{aligned}
$$
By Theorem~\ref{t:main}, there exists a surjective continuous map $f:[0,s]\to X$ such that $\w_f\le \Omega$ and hence $f$ is $\frac1\alpha$-H\"older. Then the function $g:[0,1]\to X$, $g:t\mapsto f(ts)$, is $\frac1\alpha$-H\"older,   witnessing that $\HDim(X)\le \alpha$.
\end{proof}

Given a Peano continuum $X$, let $\M_X$ be the set of all admissible metrics on $X$, i.e., metrics generating the topology of $X$. By the classical Pontrjagin--Schnirelman Theorem \cite{PS}, the number
$$\dim(X)\defeq\inf\{\Dim(X,d):d\in\M_X\}$$coincides with the topological dimension of $X$ and is integer or infinite. 

On the other hand, the topological invariants
$$\Sdim(X)\defeq\inf\{\SDim(X,d):d\in\M_X\}\quad\mbox{and}\quad\Hdim(X)\defeq\inf\{\HDim(X,d):d\in\M_X\},$$ 
are not necessarily integer: by \cite[Corollary 7.3]{BT}, for every $r\in[1,\infty]$ there exists a Peano continuum $X\subseteq\IR^2$  such that 
$$\Sdim(X)=\SDim(X)=\HoDim(X)=\Hdim(X)=r.$$
Applying Corollary~\ref{c:main} to topological  $S$-dimension $\Sdim(X)$ and the topological H\"older dimension $\Hdim(X)$, we obtain another corollary.

\begin{corollary}\label{c2} Each Peano continuum $X$ has $$\dim(X)\le \Sdim(X)\le\Hdim(X)\le 2\cdot\Sdim(X).$$
\end{corollary}

Corollary~\ref{c2} provides a partial answer to  \cite[Problem 5.2]{BT}, asking whether $\Sdim(X)=\Hdim(X)$ for any Peano continuum $X$.

\begin{remark} Corollaries~\ref{c:main} and \ref{c2} are essentially applied in the paper \cite{BNS} estabilishing the existence of a fractal structure on a Peano continuum (more generally, finite unions of Peano continua) containing a free arc and having finite $S$-dimension.% $\Sdim(X)$.
\end{remark}

\section{Preliminaries}

In this section we collect notations that will be used in the remaining part of the paper. 

We denote by $\w$ the set $\{0,1,2,\dots\}$ of all natural numbers. Each number $n\in\w$ is idenitified with the set $\{0,\dots,n-1\}$ of smaller numbers, so $\w\setminus n=\{n,n+1,\dots\}$. Let $\IN\defeq\w\setminus\{0\}$ be the set of positive integers. For a set $X$, its cardinality is denoted by $|X|$.

For a function $f:X\to Y$ between sets and a subset $A\subseteq X$, let
$$f[A]\defeq\{f(a):a\in A\}$$ be the image of the set $A$ under the map $f$.

For a metric space $X$ by $d_X$ we shall denote its metric. For a point $x\in X$ and positive real $\e>0$, let $$O(x;\e)\defeq\{y\in X:d_X(x,y)<\e\}\quad\mbox{and}\quad O[x;\e]\defeq \{y\in X:d_X(x,y)\le\e\}$$ be the {\em open $\e$-ball} and the {\em closed $\e$-ball} centered at $x$, respectively. For a subset $A\subseteq X$ let $O(A;\e)=\bigcup_{a\in A}O(a;\e)$ be the open $\e$-neighborhood around $A$, and $\overline{O(A;\e)}$ be its closure in $X$. Let also $O[A;\e]=\bigcup_{a\in A}O[a;\e]$ and $\overline {O[A;\e]}$ be the closure of $O[A;\e]$ in $X$. It is clear that $O(A;\e)\subseteq\overline{O(A;\e)}\subseteq \overline{O[A;\e]}$.
\smallskip

For a subset $A\subseteq X$ let $$\diam(A)\defeq\sup(\{d_X(x,y):x,y\in A\}\cup\{0\})$$ be the {\em diameter} of $A$ in the metric space $X$. It is easy to see that $\diam(\overline{O[A;\e]})\le\diam(A)+2\e$.

For a family $\F$ of subsets of a metric space $X$, let 
$$\mesh(\F)\defeq\sup\{\diam(F):F\in\F\}$$ be the {\em mesh} of $\F$. For a set $A\subseteq X$ and a family $\C$ of subsets of $X$, we put 
$$\C(A)\defeq\{C\in\C:\emptyset\ne C\subseteq A\}\quad\mbox{and}\quad
\C[A]\defeq\{C\in\C:C\cap A\ne\emptyset\}.$$
It is clear that $\C(A)\subseteq\C[A]$.

By an {\em ordered set} we understand a set $X$ endowed with an {\em order} $\le$  (which is a reflexive transitive antisymmetric relation on $X$). An order $\le$ on $X$ is {\em linear} if for any distinct elements $x,y\in X$ we have $x\le y$ or $y\le x$. Given two elements $x,y$ of an ordered set $X$, we write $x<y$ if $x\le y$ and $x\ne y$.

A function $f:X\to Y$ between two ordered sets $(X,\le_X)$ and $(Y,\le_Y)$ is defined to be {\em order-preserving} if for any elements $x,x'\in X$ with $x\le_X x'$  we have $f(x)\le_Y f(x')$.   A function $f$ between two ordered sets is called {\em increasing} if it is injective and order-preserving. 

A sequence $(\delta_n)_{n\in\w}$ of real numbers will be called ({\em strictly}) {\em decreasing} if for any numbers $n<m$ we have $\delta_m\le\delta_n$ (resp. $\delta_m<\delta_n$).

\section{Nested sequences of covers of metric Peano continua}\label{s:covers}

In this section we prove a lemma showing that each metric Peano continuum possesses a nice ``nested" sequence of covers by closed connected sets. Such sequences will be used in the proof of Lemma~\ref{l:path}.

\begin{lemma}\label{l:cover} For every metric Peano continuum $X$ of diameter $\le 1$ there exists a sequence $(\C_n)_{n\in\w}$ of covers of $X$ by closed connected subsets such that $\C_0=\{X\}$ and for every $n\in\IN$ the following conditions are satisfied:
\begin{enumerate}
\item $|\C_n|\le S_X(\frac1{2^n})$;
\item $\mesh(\C_n)\le \frac3{2^n}$;
\item $C=\bigcup\C_n(C)$ for every $C\in\C_{n-1}$.
\end{enumerate}
\end{lemma} 

\begin{proof} For every $n\in\w$ the definition of the number $S_X(2^{-n})$ yields a finite cover $\F_n$ of $X$ by  connected subsets such that $|\F_n|=S_X(2^{-n})$ and $\mesh(\F_n)\le 2^{-n}$. Since $\diam\,X\le 1$, $S_X(2^0)=1$ and hence $\F_0=\{X\}$.

For every set $A\subseteq X$ and number $n\in\w$, let $\St_{n}[A]\defeq\bigcup\F_n[A]$, where $\F_n[A]\defeq\{F\in\F_n:F\cap A\ne\emptyset\}$. Since $\F_n$ is a cover of $X$, $A\subseteq \St_n[A].$ 

It is clear that 
\begin{equation}\label{eq:bound1}
\diam(\St_n[A])\le \diam(A)+2\,\mesh(\F_n)\le\diam(A)+\tfrac2{2^n}.
\end{equation}

\begin{claim}\label{cl:connected} For any $n\in\w$ and connected subspace $A\subseteq X$, the subspace $\St_n[A]$ of $X$ is connected.
\end{claim}

\begin{proof} This claim follows from \cite[Theorem 6.1.9]{E} and the observation that $$\textstyle\St_n[A]=A\cup\bigcup\{F\in\F_n:A\cap F\ne\emptyset\}$$ is the union of connected subspaces that intersect the connected space $A$.
\end{proof}

Let $\St_{n,n}[A]\defeq\St_n[A]$ and $\St_{n,m}[A]\defeq\St_m[\St_{n,m-1}[A]]$ for $m>n$. Let also $$\St_{n,\w}[A]\defeq\bigcup_{m=n}^\infty\St_{n,m}[A].$$

\begin{claim}\label{cl:connected2} For every numbers $n\le m$ and connected subspace $A\subseteq X$ the subspaces $\St_{n,m}[A]$, $\St_{n,\w}[A]$ and $\overline{\St_{n,\w}[A]}$ of $X$ are connected.
\end{claim}

\begin{proof} This claim will be proved by induction on $m\ge n$. For $m=n$ the space $\St_{n,m}[A]=\St_n[A]$ is connected by Claim~\ref{cl:connected}. Assume that for some $m\ge n$ we have proved that the subspace $\St_{n,m}[A]$ of $X$ is connected. Applying Claim~\ref{cl:connected}, we conclude that the subspace $\St_{n,m+1}[A]=\St_{m+1}[\St_{n,m}[A]]$ of $X$ is connected. The space $\St_{n,\w}[A]=\bigcup_{m=n}^\infty\St_{n,m}[A]$ is connected, being the union of an increasing sequence of connected spaces, see \cite[Corollary~6.1.10]{E}. By \cite[Corollary 6.1.11]{E}, the closure $\overline{\St_{n,\w}[A]}$ of the connected subset $\St_{n,\w}[A]$ in $X$ is connected.
\end{proof}

The upper bound (\ref{eq:bound1}) implies the upper bound 
\begin{equation}\label{eq:bound2}
\diam(\St_{n,\w}[A])\le\diam(A)+2\sum_{m=n}^\infty\mesh(\F_m)\le\diam(A)+2\sum_{m=n}^\infty\tfrac1{2^m}=\diam(A)+\tfrac4{2^n}.
\end{equation}

For every $n\in\w$, consider the closed cover $\C_n\defeq\{\overline{\St_{n+1,\w}[F]}:F\in\F_n\}$. By Claim~\ref{cl:connected2}, the family $\C_n$ consists of closed connected subspaces of $X$. It follows that $|\C_n|\le|\F_n|=S_X(2^{-n})$ and $$\mesh(\C_n)\le\mesh(\F_n)+\tfrac4{2^{n+1}}\le\tfrac1{2^n}+\tfrac2{2^n}=\tfrac3{2^n}.$$
This shows that the family $\C_n$ has the properties (1), (2) of Lemma~\ref{l:cover}.
The property (3) will be derived from the following claims.

\begin{claim}\label{cl:st-mono} For any numbers $n\le m\le k$ and sets $A\subseteq B\subseteq X$ we have $$\St_{n,m}[A]\subseteq \St_{n,m}[B]\subseteq \St_{n,k}[B]\quad\mbox{and}\quad \St_{n,\w}[A]\subseteq\St_{n,\w}[B]=\bigcup_{i=m}^\infty\St_{n,i}[B].$$
\end{claim}

\begin{proof} The inclusion $\St_{n,m}[A]\subseteq\St_{n,m}[B]$ will be proved by induction on $m\ge n$. For $m=n$ we have $\F_n[A]\subseteq\F_n[B]$ and hence $$\textstyle\St_{n,n}[A]=\bigcup\F_n[A]\subseteq \bigcup\F_n[B]=\St_{n,n}[B].$$
Assume that for some $m\ge n$ we have proved that $\St_{n,m}[A]\subseteq\St_{n,m}[B]$. Then $$\St_{n,m+1}[A]=\St_{m+1}[\St_{n,m}[A]]\subseteq \St_{m+1}[\St_{n,m}[B]]=\St_{n,m+1}[B],$$ which completes the inductive step.

The inclusion $\St_{n,m}[B]\subseteq\St_{n,k}[B]$ can be proved by induction on $k\ge m$. For $k=m$ this inclusion is trivial. Assuming that $\St_{n,m}[B]\subseteq\St_{n,k}[B]$ for some $k\ge m$, we conclude that $$\St_{n,m}[B]\subseteq\St_{n,k}[B]\subseteq\St_{k+1}[\St_{n,k}[B]]=\St_{n,k+1}[B].$$
%. Given any $x\in\St_{n,m+1}[A]=\St_{m+1}[\St_{n,m}[A]]$, find a set $F\in\F_{m+1}$ such that $x\in F$ and $F\cap\St_{n,m}[A]\ne\emptyset$. Taking into account that $\St_{n,m}[A]\subseteq\St_{n,m}[B]$, we conclude that $\emptyset\ne F\cap \St_{n,m}[A]\subseteq F\cap\St_{n,m}[B]$ and hence $x\in F\subseteq\St_{m+1}[\St_{n,m}[B]]=\St_{n,m+1}[B]$. This completes the proof of the inclusion $\St_{n,m+1}[A]\subseteq\St_{n,m+1}[B]$. 

It follows that $\St_{n,\w}[A]=\bigcup_{i=n}^\infty\St_{n,i}[A]\subseteq\bigcup_{i=n}^\infty\St_{n,i}[B]=\St_{n,\w}[B]=\bigcup_{i=m}^\infty\St_{n,i}[B]$ because $\bigcup_{i=n}^m\St_{n,i}[B]=\St_{n,m}[B]$.
\end{proof}

\begin{claim}\label{cl:st-algebra} For every numbers $n<m\le k$ and every set $A\subseteq X$, we have $$\St_{n,k}[A]=\St_{m,k}[\St_{n,m-1}[A]]\quad\mbox{and}\quad\St_{n,\w}[A]=\St_{m,\w}[\St_{n,m-1}[A]].$$
\end{claim}

\begin{proof} The equality $\St_{n,k}[A]=\St_{m,k}[\St_{n,m-1}[A]]$ will be proved by induction on $k\ge m$. For $k=m$ we have $\St_{n,k}[A]=\St_{n,m}[A]=\St_{m}[\St_{n,m-1}[A]]=\St_{m,m}[\St_{n,m-1}[A]]=\St_{m,k}[\St_{n,m-1}[A]]$ by the definition of the set $\St_{n,m}[A]$. Assume that for some $k\ge m$ we have proved that  $\St_{n,k}[A]=\St_{m,k}[\St_{n,m-1}[A]]$. Then $$\St_{n,k+1}[A]=\St_{k+1}[\St_{n,k}[A]]=\St_{k+1}[\St_{m,k}[\St_{n,m-1}[A]]]=\St_{m,k+1}[\St_{n,m-1}[A]].$$ This completes the inductive step and also the proof of the equality $\St_{n,k}[A]=\St_{m,k}[\St_{n,m-1}[A]]$ for all $k\ge m$.

These equalities imply $$\St_{n,\w}[A]=\bigcup_{k=n}^\infty\St_{n,k}[A]=\bigcup_{k=m}^\infty\St_{n,k}[A]=\bigcup_{k=m}^\infty\St_{m,k}[\St_{n,m-1}[A]]=\St_{m,\w}[\St_{n,m-1}[A]].$$
\end{proof}

\begin{claim}\label{cl:ind} For every positive numbers $n<m$, every set $F_{n-1}\in\F_{n-1}$ and point $x_m\in \St_{n,m}[F_{n-1}]$ there exists a set $F_n\in\F_n$ such that $x_m\in \St_{n+1,m}[F_{n}]\subseteq \St_{n,m}[F_{n-1}]$.
\end{claim}

\begin{proof} This claim will be proved by induction on $m$. For $m=n+1$ we have $$\St_{n,m}[F_{n-1}]=\St_{n,n+1}[F_{n-1}]=\St_{n+1}[\St_{n,n}[F_{n-1}]]=\St_{n+1}[\St_n[F_{n-1}]]$$ and hence for any point $x_{m}\in\St_{n,m}[F_{n-1}]$ there exist sets $F_{n+1}\in\F_{n+1}$ and $F_n\in\F_n$ such that $x_{m}\in F_{n+1}$ and  $F_{n+1}\cap F_n\ne\emptyset\ne F_n\cap F_{n-1}$. It follows from $F_n\subseteq \St_{n}[F_{n-1}]$ that  $$\St_{n+1,m}[F_n]=\St_{n+1}[F_n]\subseteq\St_{n+1}[\St_n[F_{n-1}]]=\St_{n,n+1}[F_{n-1}]=\St_{n,m}[F_{n-1}]$$ and hence $x_m\in F_{n+1}\subseteq \St_{n+1}[F_n]=\St_{n+1,m}[F_n]\subseteq \St_{n,m}[F_{n-1}]$, so the set $F_n$ has the desired property.

Now assume that for some $m>n$ and every $x_m\in\St_{n,m}[F_{n-1}]$ we can find a set $F_n\in\F_n$ such that  $x_m\in \St_{n+1,m}[F_{n}]\subseteq \St_{n,m}[F_{n-1}]$. Given any point $x_{m+1}\in\St_{n,m+1}[F_{n-1}]=\St_{m+1}[\St_{n,m}[F_{n-1}]]$, find a set $F_{m+1}\in\F_{m+1}$ such that $x_{m+1}\in F_{m+1}$ and the intersection $F_{m+1}\cap\St_{n,m}[F_{n-1}]$ contains some point $x_m$. By the inductive assumption, there exists $F_n\in\F_n$ such that $x_m\in\St_{n+1,m}[F_{n}]\subseteq\St_{n,m}[F_{n-1}]$. By Claim~\ref{cl:st-mono}, 
$$\St_{n+1,m+1}[F_n]=\St_{m+1}[\St_{n+1,m}[F_{n}]]\subseteq \St_{m+1}[\St_{n,m}[F_{n-1}]]=\St_{n,m+1}[F_{n-1}]$$ and hence $$x_{m+1}\in F_{m+1}\subseteq\St_{m+1}[\{x_m\}]\subseteq\St_{m+1}[\St_{n+1,m}[F_n]]= \St_{n+1,m+1}[F_n]\subseteq\St_{n,m+1}[F_{n-1}].$$ This completes the inductive step and also the inductive proof of Claim~\ref{cl:ind}.
\end{proof}

\begin{claim}\label{cl:st-union} For every number $n\in\IN$, set $F_{n-1}\in\F_{n-1}$ and point $x\in \St_{n,\w}[F_{n-1}]$, there exists a set $F_n\in\F_n$ such that $x\in \St_{n+1,\w}[F_n]\subseteq \St_{n,\w}[F_{n-1}]$. 
\end{claim}

\begin{proof} Since $\St_{n,\w}[F_{n-1}]=\bigcup_{m=n+1}^\infty\St_{n,m}[F_{n-1}]$, there exists a number $m>n$ such that $x\in\St_{n,m}[F_{n-1}]$. By Claim~\ref{cl:ind}, there exists a set $F_n\in\F_n$ such that $x\in\St_{n+1,m}[F_n]\subseteq \St_{n,m}[F_{n-1}]$. Claims~\ref{cl:st-mono} and \ref{cl:st-algebra} ensure that 
$$x\in \St_{n+1,m}[F_n]\subseteq \St_{n+1,\w}[F_n]=\St_{m+1,\w}[\St_{n+1,m}[F_n]]\subseteq \St_{m+1,\w}[\St_{n,m}[F_{n-1}]]=\St_{n,\w}[F_{n-1}].$$
\end{proof}

Claim~\ref{cl:st-union} implies that for every $n\in\IN$ and every set $F_{n-1}\in\F_{n-1}$, the set $\St_{n,\w}[F_{n-1}]$ is the union of the {\em finite} family $\F\defeq\{\St_{n+1,\w}[F]:F\in\F_n,\;\St_{n+1,\w}[F]\subseteq\St_{n,\w}[F_{n-1}]\}$. Since $\F\subseteq\{\St_{n+1,\w}[F]:F\in\F_n,\;\overline{\St_{n+1,\w}[F]}\subseteq\overline{\St_{n,\w}[F_{n-1}]}\}$, we obtain the inclusions 
$$
\begin{aligned}
\overline{\St_{n,\w}[F_{n-1}]}&=\bigcup\{\overline{\St_{n+1,\w}[F]}:F\in\F_n,\;\St_{n+1,\w}[F]\subseteq\St_{n,\w}[F_{n-1}]\}\subseteq\\
&\subseteq\bigcup\{\overline{\St_{n+1,\w}[F]}:F\in\F_n,\;\overline{\St_{n+1,\w}[F]}\subseteq\overline{\St_{n,\w}[F_{n-1}]}\}=\\
&=\bigcup\{C\in\C_n:C\subseteq\overline{\St_{n,\w}[F_{n-1}]}\}\subseteq
\overline{\St_{n,\w}[F_{n-1}]},
\end{aligned}
$$
establishing the property (3) of the sequence $(\C_n)_{n\in\w}$ in Lemma~\ref{l:cover}.
\end{proof}

\section{Minimal connectors and their orderings}

Let $A,B$ be two disjoint  nonempty subsets of a topological space $X$. 

\begin{definition} A family $\C$ of connected subsets of $X$ is called 
\begin{itemize}
\item an {\em $(A,B)$-connector} if the union $\bigcup\C$ is connected and $A\cap\bigcup\C\ne\emptyset\ne B\cap\bigcup\C$;
\item a {\em minimal $(A,B)$-connector} if every $(A,B)$-connector $\C'\subseteq \C$ coincides with $\C$.
\end{itemize}
\end{definition}

It is clear that every finite $(A,B)$-connector contains a minimal $(A,B)$-connector.

Next, we define an ordered counterpart of a minimal $(A,B)$-connector. 

\begin{definition}\label{d:chain} A finite sequence $(C_1,\dots,C_n)$ of subsets of $X$ is called an {\em $(A,B)$-chain} if for every $i,j\in\{1,\dots,n\}$ the following conditions are satisfied:
\begin{enumerate}
\item $A\cap C_i\ne\emptyset$ if and only if $i=1$;
\item $B\cap C_j\ne\emptyset$ if and only if $j=n$;
\item $C_i\cap C_j\ne\emptyset$ if and only if $|i-j|\le 1$.
\end{enumerate}
\end{definition}

\begin{lemma}\label{l:connector} Let $A,B$ be disjoint nonempty sets in a topological space $X$. A finite family $\C$ of closed connected subsets of $X$ is a minimal $(A,B)$-connector $\C$ if and only if $\C=\{C_1,\dots,C_n\}$ for a unique $(A,B)$-chain $(C_1,\dots,C_n)$.
\end{lemma}

\begin{proof} To prove the ``if'' part, assume that $\C=\{C_1,\dots,C_n\}$ for some $(A,B)$-chain $(C_1,\dots,C_n)$, satisfying the conditions (1)--(3) of Definition~\ref{d:chain}. The connectedness of the sets $C_i\in\C$ and the condition (3)  imply that the union $\bigcup_{i=1}^nC_i=\bigcup\C$ is connected. The conditions (1) and (2)  imply that $\C$ is an $(A,B)$-connector. Assuming that the $(A,B)$-connector $\C$ is not minimal, we can find an $(A,B)$-connector $\C'\subset \C$ and a number $k\in\{1,\dots,n\}$ such that $C_k\notin \C'$. The conditions (1) and (2) ensure that $1<k<n$ and $C_1,C_n\in\C'$. The condition (3) implies that $(\bigcup\C')\cap\bigcup_{i=1}^{k-1}C_i$ and $(\bigcup\C')\cap\bigcup_{i=k+1}^nC_i$ are two disjoint closed nonempty sets in $\bigcup\C'$, which contradicts the connectedness of $\bigcup\C'$. This contradiction shows that $\C$ is a minimal $(A,B)$-connector.
\smallskip

To prove the ``only if'' part, assume that $\C$ is a minimal $(A,B)$-connector. By the minimality of $\C$, each set $C\in\C$ is not empty, moreover, $C\not\subseteq \bigcup(\C\setminus\{C\})$. Let $\C(A)\defeq\{C\in\C:C\subseteq A\}$ and let $\C'\subseteq\C$ be the family of all sets $C\in\C$ for which there exist $x\in X\setminus A$, $k\in\IN$, and an $(A,\{x\})$-chain $(C_1,\dots,C_k)\in\C^k$ such that $C=C_k$. Since the family $\C$ is finite and consists of closed subsets of $X$, the sets $\bigcup\C(A)$ and $\bigcup\C'$ are closed in $\bigcup\C$. 

We claim that the family $\C'_A\defeq\C(A)\cup\C'$ is not empty. Since $A\cap\bigcup\C\ne\emptyset$, there exists a set $C\in\C$ such that $A\cap C\ne\emptyset$. If $C\subseteq A$, then $C\in\C(A)\subseteq\C'_A$. If $C\not\subseteq A$, then there exists a point $x\in C\setminus A$ and $(C)$ is an $(A,\{x\})$-chain witnessing that $C\in\C'\subseteq\C'_A$.

Next, we show that $\C'_A=\C$. Assuming that $\C'_A\ne\C$,  consider the nonempty family $\C''=\C\setminus\C'_A$ and observe that its union $\bigcup\C''$ is a closed nonempty subset of $\bigcup\C=(\bigcup\C'_A)\cup(\bigcup\C'')$. The connectedness of $\bigcup\C$ implies that $(\bigcup\C'_A)\cap(\bigcap\C'')\ne\emptyset$ and hence there exist  set $C'\in\C'_A$ and $C''\in\C''$ such that $C'\cap C''\ne\emptyset$. If $C''\cap A\ne\emptyset$, then $(C'')$ is an $(A,\{x\})$-chain for any $x\in C''\setminus A$ (which exists as $C''\in\C''\subseteq\C\setminus\C(A)$) and hence $C''\in\C'$, which contradicts the choice of $C''\in\C''\subseteq\C\setminus\C'$. So, $C''\cap A=\emptyset\ne C'\cap C''$ and hence $C'\in\C'_A\setminus\C(A)\subseteq \C'$. Then there exist $x\in X\setminus A$, $k\in\IN$, and an $(A,\{x\})$-chain $(C_1,\dots,C_k)\in\C^k$ such that $C_k=C'$. Let $j\in\{1,\dots,k\}$ be the smallest number such that $C_j\cap C''\ne\emptyset$.

We claim that $C_j\ne C''$. Assuming that $C_j=C''$ and taking into account that $C''\cap A=\emptyset\ne C_1\cap A$, we conclude that $j\ne 1$ and hence the number $j-1$ is well-defined. Since $\emptyset\ne C_{j-1}\cap C_j=C_{j-1}\cap C''$, we obtain a contradiction with the choice of the number $j$. Therefore, $C_j\ne C''$ and the minimality of the $(A,B)$-connector $\C$ ensures that $C''\not\subseteq C_j$, so we can choose a point $x\in C''\setminus C_j=C''\setminus\bigcup_{i=1}^jC_i$ and conclude that $(C_1,\dots,C_j,C'')$ is an $(A,\{x\})$-chain, witnessing that $C''\in\C'$, which contradicts the choice of $C''\in\C\setminus\C'$. This contradiction shows that $\C'_A=\C$. 

Since $B\cap\bigcup\C\ne\emptyset$, there exists a set $C\in\C$ such that $C\cap B\ne\emptyset$. Then $C\in\C\setminus\C(A)=\C_A'\setminus\C(A)\subseteq\C'$ and hence there exist $x\in X\setminus A$, $k\in\IN$, and an $(A,\{x\})$-chain $(C_1,\dots,C_k)$ such that $C_k=C$. Let $j\in\{1,\dots,k\}$ be the smallest number such that $C_j\cap B\ne\emptyset$. Then $(C_1,\dots,C_j)$ is an $(A,B)$-chain and $\{C_1,\dots,C_j\}=\C$ by the minimality of the $(A,B)$-connector $\C$. 

It remains to prove that an $(A,B)$-chain $(C_1,\dots,C_n)$ with $\{C_1,\dots,C_n\}=\C$ is unique. Assume that $(C_1',\dots,C_m')$ is another $(A,B)$-chain such that $\C=\{C_1',\dots,C_m'\}$. The definition of an $(A,B)$-chain ensures that it consists of pairwise distinct sets, which implies that $n=|\C|=m$. The definition of an $(A,B)$-chain ensures that $C_1=C_1'$ is the unique set in the family $\C$ that intersects $A$. The set $C_2=C_2'$ is a unique set in $\C\setminus\{C_1\}$ that intersects the set $C_1$. Assume that for some $k\in\{2,\dots,n-1\}$ we have proved that $C_i=C_i'$ for all $i\le k$. The definition of an $(A,B)$-chain ensures that $C_{k+1}=C_{k+1}'$ is a unique set in $\C$ that intersects $C_k$ and is distinct from $C_{k-1}=C_{k-1}'$. Therefore, $C_i=C_i'$ for all $i\in\{1,\dots,n\}$, and the $(A,B)$-chain $(C_1,\dots,C_n)$ representing the minimal $(A,B)$-connector $\C$ is unique.
\end{proof}

Lemma~\ref{l:connector} implies that every finite minimal $(A,B)$-connector $\C$  carries a unique linear order $\preceq$ such that for any sets $C\preceq C'$ in $\C$ the following conditions hold:
\begin{itemize}
\item[(i)] $C\cap A\ne\emptyset$ if and only if $C$ is the smallest element of the linearly ordered set $(\C,\preceq)$;
\item[(ii)] $C\cap B\ne \emptyset$ if and only if $C$ is the largest element of the linearly ordered set $(\C,\preceq)$; 
\item[(iii)] $C\cap C'\ne\emptyset$ if and only if $\{M\in\C:C\preceq M\preceq C'\}=\{C,C'\}$.
\end{itemize}
The linear order $\preceq$ will be called the {\em canonical linear order} on the minimal $(A,B)$-connector $\C$.

\section{Paths with controlled continuity modulus in metric Peano continua}

In this section we prove some results on paths with controlled continuity modulus in metric Peano continua. The key result here is the following technical lemma, which is one of two ingredients of the proof of the controlled Hahn--Mazurkiewicz Theorem~\ref{t:main}.

\begin{lemma}\label{l:path} Let $(\delta_n)_{n\in\w}$ be a decreasing sequence of positive real numbers and $X$ be a compact connected metric space of diameter $\le 1$ such that $\sum_{n=0}^\infty S_X(\tfrac1{2^{n}})\delta_n<\infty$. For every connected subset $F\subseteq X$, points $a,b\in F$ and number $m\in\w$, there exist a positive real number $s\le \sum_{n=m}^\infty S_X(\tfrac1{2^{n}})\delta_n$ and a function $p:[0,s]\to O[F;\tfrac{3}{2^m}]$ such that $p(0)=a$, $p(s)=b$, and for every number $n\ge m$ and points $x,y\in[0,s]$ with $|x-y|<\delta_n$ we have $d_X(p(x),p(y))\le\frac6{2^n}$.
\end{lemma} 

\begin{proof} Fix a connected set $F\subseteq X$, points $a,b\in F$ and a number $m\in\w$. If $a=b$, then we can put $s=0$ and $p:[0,s]\to \{a\}=\{b\}$ be the constant map. So, assume that $a\ne b$. 

By Lemma~\ref{l:cover}, there exists a sequence $(\C_n)_{n\in\w}$ of covers of $X$ by closed connected subsets such that $\C_0=\{X\}$ and for every $n\in\IN$ the following conditions are satisfied:
\begin{enumerate}
\item $|\C_n|\le S_X(2^{-n})$;
\item $\mesh(\C_n)\le \frac3{2^n}$;
\item $C=\bigcup\C_n(C)$ for every $C\in\C_{n-1}$.
\end{enumerate}

\begin{claim}\label{cl:M} There exist a sequence $(\M_n)_{n=m}^\infty$ of minimal $(\{a\},\{b\})$-connectors $\M_n\subseteq\C_n$ and a double sequence $(\mu_k^n)_{m\le k\le n}$ of surjective maps $\mu_k^n:\M_n\to\M_k$ such that $\M_m\subseteq\C_m[F]$ and for every numbers $k,n\in\w$ with  $m\le k\le n$ the following conditions are satisfied:
\begin{enumerate}
\item[(4)] If $n>m$, then $\bigcup\M_{n}\subseteq\bigcup\M_{n-1}$;
\item[(5)] $\mu_{k}^{n}:\M_{n}\to\M_k$ is an order-preserving map of the $(\{a\},\{b\})$-connectors $\M_n$ and $\M_k$ endowed with their canonical linear orders;
\item[(6)] $M\subseteq \mu_k^{n}(M)$ for every $M\in\M_{n}$.
\end{enumerate}
\end{claim}

\begin{proof} The sequences $(\M_n)_{n=m}^\infty$ and $(\mu_k^{n})_{m\le k\le n}$ will be constructed by induction on $n\ge m$. To start the inductive construction, observe that the family $\C_m[F]=\{C\in \C_m:C\cap F\neq\emptyset\}$ is an $(\{a\},\{b\})$-connector and hence it contains some minimal $(\{a\},\{b\})$-connector $\M_m\subseteq\C_m[F]$. Let $\mu_m^m:\M_m\to\M_m$ be the identity map of $\M_m$. Assume that for some $n\ge m$ we have defined sequences $(\M_k)_{k=m}^n$ and $(\mu_k^n)_{m\le k\le n}$ satisfying the inductive conditions.% (4)--(6). 

By Lemma~\ref{l:connector}, for the minimal $(\{a\},\{b\})$-connector $\M_n$ there exists a unique $(\{a\},\{b\})$-chain $(M_{1},\dots,M_{|\M_n|})$ such that $\M_n=\{M_{1},\dots, M_{|\M_n|}\}$. By the condition (3), for every $M\in\M_n$ the family $\C_{n+1}(M)\defeq\{C\in\C_{n+1}:C\subseteq M\}$ has connected union $\bigcup\C_{n+1}(M)=M\in\C_n$. 

Inductively for every $i\in\{1,\dots,|\M_n|\}$ we shall construct two disjoint nonempty sets $A_i,B_i$ in $M_i$ and a minimal $(A_i,B_i)$-connector $\A_i\subseteq \C_{n+1}(M_i)$ such that the following conditions are satisfied:
\begin{enumerate}
\item[(7)] if $i=1$, then $A_i=\{a\}$;
\item[(8)] if $i>1$, then $A_i=M_i\cap\max\A_{i-1}$;
\item[(9)] $B_i=M_i\cap M_{i+1}$, where $M_{|\M_n|+1}=\{b\}$.
\end{enumerate}
In the inductive condition (8) by $\max \A_{i-1}$ we denote the maximal element of the $(A_{i-1},B_{i-1})$-connector $\A_{i-1}$ endowed with its canonical linear order. We start the inductive construction putting $A_1=\{a\}\subset M_1$ and 
$B_1=M_1\cap M_2$ where $M_{|\M_n|+1}=\{b\}$. 
Since $(M_1,\dots,M_{|\M_n|})$ is an $(\{a\},\{b\})$-chain, the sets $A_1$ and $B_1$ are disjoint.  Since the family $\C_{n+1}(M_1)$ is a finite $(A_1,B_1)$-connector, there exists a minimal $(A_1,B_1)$-connector $\A_1\subseteq\C_{n+1}(M_1)$. 

Now assume that for some $i\le |\M_n|$ we have constructed disjoint nonempty sets $A_{i-1},B_{i-1}\subseteq M_{i-1}$ and a minimal $(A_{i-1},B_{i-1})$-connector $\A_{i-1}\subseteq\C_{n+1}(M_{i-1})$. The condition (9) ensures that $B_{i-1}=M_{i-1}\cap M_{i}$. Observe that for the maximal element $\max\A_{i-1}$ of the minimal $(A_{i-1},B_{i-1})$-connector $\A_{i-1}$ endowed with its canonial linear order $\preceq_{i-1}$, the intersection  $B_{i-1}\cap\max\A_{i-1}=M_{i-1}\cap M_{i}\cap\max\A_{i-1}=M_{i}\cap\max\A_{i-1}={\!\!:}\,A_{i}$ is nonempty. Since the sets $M_{i-1}$ and $M_{i+1}$ are disjoint, the sets $A_{i}\subseteq \max\A_{i-1}\subseteq M_{i-1}$ and $B_{i}\defeq M_{i}\cap M_{i+1}$ are disjoint, too. Since the family $\C_{n+1}(M_{i})$ has connected union $\bigcup\C_{n+1}(M_i)=M_i\supseteq A_i\cup B_i$, it is a finite $(A_{i},B_i)$-connector, containing a minimal $(A_i,B_i)$-connector $\A_i\subseteq\C_{n+1}(M_i)$. This completes the inductive step of the construction of the sequence $(\A_i)_{i=1}^{|\M_n|}$.
\smallskip

After completing the inductive construction, consider the family $\M_{n+1}\defeq\bigcup_{i=1}^{|\M_n|}\A_i$ and applying Lemma~\ref{l:connector}, show that it is a minimal $(\{a\},\{b\})$-connector. Let $\mu^{n+1}_{n+1}:\M_{n+1}\to\M_{n+1}$ be the identity function of $\M_{n+1}$ and let $\mu_n^{n+1}:\M_{n+1}\to \M_n$ be the function such that $\mu_n^{n+1}[\A_i]=\{M_i\}$ for every $i\in\{1,\dots,|\M_n|\}$.
It is easy to see that the function $\mu^{n+1}_n:\M_{n+1}\to\M_n$ is order-preserving, and $M\subseteq\mu_n^{n+1}(M)$ for every $M\in\M_{n+1}$. For every $k\in\w$ with $m\le k<n$, let $\mu_k^{n+1}\defeq\mu^n_k\circ \mu^{n+1}_n$. Applying the inductive conditions (5), (6), show that the function $\mu^{n+1}_k:\M_{n+1}\to\M_k$ is order-preserving, and $M\subseteq\mu_k^{n+1}(M)$ for every $M\in\M_{n+1}$.\end{proof}

Let $(\M_n)_{n=m}^\infty$ and $(\mu_k^{n})_{m\le k\le n}$ be the sequences constructed in Claim~\ref{cl:M}.
For every $n\ge m$, let $\preceq_n$ be the canonical linear order on the $(\{a\},\{b\})$-connector $\M_n\subseteq\C_n$. %Given two sets $M,M'\in\M_n$ we shall write $M\prec_n M'$ iff $M\preceq_n M'$ and $M\ne M'$. 
The condition (4) of Claim~\ref{cl:M} ensures that  the sequence $(\bigcup\M_n)_{n=m}^\infty$ is decreasing.
By \cite[Theorem 6.1.18]{E}, the intersection $I=\bigcap_{n=m}^\infty \bigcup\M_n$ is a compact connected set containing the points $a,b$. Observe that
$$\textstyle I=\bigcap_{n=m}^\infty\bigcup\M_n\subseteq\bigcup\M_m\subseteq\bigcup\C_m[F]\subseteq O[F;\mesh(\C_m)]\subseteq O[F;\tfrac3{2^m}].$$

On the set $I$ consider the relation $\preceq$ defined by $x\preceq y$ iff for every  number $n\ge m$ with $\frac6{2^n}<d_X(x,y)$ and every $M_x,M_y\in \M_n$ with $x\in M_x$ and $y\in M_y$ we have $M_x\preceq_n M_y$. %This definition 

\begin{claim}\label{cl:prec-I} For two points $x,y\in I$ the relation $x\prec y$ holds if and only if there exist $n\ge m$ with $\frac6{2^n}<d_X(x,y)$ and sets $M_x,M_y\in\M_n$ such that  $x\in M_x$, $y\in M_y$ and $M_x\prec_n M_y$.
\end{claim}

\begin{proof} The ``only if'' part of this characterization follows from the definition of the relation $\preceq$ on $I$. To prove the ``if'' part, assume that for two points $x,y\in I$  there exist $n\ge m$ and sets $M_x,M_y\in\M_n$ such that $\frac6{2^n}<d_X(x,y)$, $x\in M_x$, $y\in M_y$ and $M_x\prec_n M_y$. To see that $x\prec y$, take any $k\ge m$ with $\frac6{2^k}<d_X(x,y)$ and any sets $M'_x,M_y'\in\M_k$ with $x\in M'_x$ and $y\in M'_y$. We should prove that $M'_x\preceq_k M_y'$. Let $l\defeq\min\{n,k\}$ and $L_x\defeq \mu^n_l(M_x)$, $L_y\defeq\mu^n_l(M_y)$, $L_x'\defeq\mu_l^k(M_x')$ and $L_y'\defeq\mu^k_l(M_y')$. The condition (6) of Claim~\ref{cl:M} ensures that $x\in M_x\cap M_x'\subseteq L_x\cap L_x'$ and $y\in M_y\cap M_y'\subseteq L_y\cap L_y'$. The order-preserving property of the function $\mu^n_l$ implies $L_x\preceq_l L_y$. Since $\diam(L_x)+\diam(L_y)\le 2{\cdot}\mesh(\M_l)\le2{\cdot}\mesh(\C_l)\le \frac6{2^l}=\max\{\frac6{2^n},\frac6{2^k}\}<d_X(x,y)$, the sets $L_x$ and $L_y$ are disjoint and hence the set $\{M\in\M_l:L_x\prec_l M\prec_l L_y\}$ is not empty. Then $L_x\cap L_x'\ne\emptyset\ne L_y\cap L_y'$ implies $L_x'\preceq L'_y$.  Since $\diam(L'_x)+\diam(L'_y)\le 2{\cdot}\mesh(\M_l)\le2{\cdot}\mesh(\C_l)<d_X(x,y)$, the sets $L'_x$ and $L'_y$ are distinct and hence $L_x'\prec_l L_y'$. The order-preserving property of the function $\mu_l^k$ ensures that $M'_x\prec_k M_y'$, witnessing that $x\prec y$.
\end{proof}

\begin{claim} The relation $\preceq$ is a linear order on $I$.
\end{claim}

\begin{proof} We need to check that the relation $\preceq$ is reflexive, antisymmetric, transitive, and linear.
\smallskip

1. The definition of $\preceq$ implies that $x\preceq x$ for every $x\in I$, so the relation $\preceq$ is reflexive.
\smallskip

2. To show that the relation $\preceq$ is antisymmetric, take any  elements $x,y\in I$ with $x\preceq y$ and $y\preceq x$. Assuming that $x\ne y$, we can choose a number $n\ge m$ with $\frac{6}{2^n}<d_X(x,y)$ and  sets $M_x,M_y\in\M_n$ with $x\in M_x$ and $y\in M_y$. The inequalities $x\preceq y\preceq x$ imply $M_x\preceq_n M_y\preceq_n M_x$ and hence $M_x=M_y$, by the antisymmetry of the order $\preceq_n$ on $\M_n$. Then $d_X(x,y)\le\diam(M_x)\le\mesh(\M_n)\le\mesh(\C_n)\le \frac3{2^n}<\frac6{2^n}<d_X(x,y)$, which is a contradiction showing that $x=y$.
\smallskip

3. To show that the reflexive relation $\preceq$ is transitive, it suffices to show that $x\preceq z$ for any distinct points $x,y,z\in I$ with $x\preceq y\preceq z$. Take any number $k\ge m$ with 
$$\tfrac6{2^k}<\min\{d_X(x,y),d_X(y,z),d_X(x,z)\}$$ and any sets $M_x,M_y,M_z\in\M_k$ with $x\in M_x$, $y\in M_y$, $z\in M_z$. It follows from 
$x\preceq y\preceq z$ that $M_x\preceq_k M_y\preceq_k M_z$ and hence $M_x\preceq_k M_z$ by the transitivity of the  order $\preceq_k$. Applying Claim~\ref{cl:prec-I}, we conclude that $x\preceq z$.
%By the definition of the relation Since $x\ne y\ne z$, there exists a number $n\ge k$ such that $\frac3{2^n}<\min\{d_X(x,y),d_X(y,z)\}$. Choose any sets $M'_x,M'_y,M'_z\in\M_n$ such that $x\in M_x'$, $y\in M'_y$, $z\in M'_z$. It follows from $x\preceq y\preceq z$ that $M'_x\preceq_n M_y'\preceq_n M_z'$ and hence $M_x'\preceq_n M_z'$, by the transitivity of the order $\preceq_n$. Let $M_x''=\mu^n_k(M_x')$ and $M_z''=\mu^n_k(M_z')$. The order-preserving property of the function $\mu^n_k:\M_n\to\M_k$ implies that $M''_x\preceq_k M''_z$. The condition (6) of Claim~\ref{cl:M} ensures that $x\in M'_x\subseteq \mu^n_k(M_x')=M_x''$ and  $z\in M'_z\subseteq \mu^n_k(M_z')=M_z''$. Then $M_x\cap M_x''\ne\emptyset\ne M_z\cap M_z''$. It follows from $\diam(M''_x)+\diam(M_z'')\le 2\,\mesh(\M_k)\le 2\,\mesh(\C_k)\le\frac6{2^k}<d_X(x,z)$ that $M_x''\cap M_z''=\emptyset$ and hence the order interval $\{M\in\M_k:M_x''\prec_k M\prec_k M_z''\}$ is nonempty.\footnote{FS here asks something. TB: Exactly.} Taking into account that $M_x\cap M_x''\ne\emptyset\ne M_z\cap M_z''$, we conclude that $M_x\preceq_k M_z$ and hence $x\preceq z$.
\smallskip

4. To show that the partial order $\preceq$ is linear, take any distinct elements $x,y\in I$. Choose any number $k\ge m$ with $\frac6{2^k}\le d_X(x,y)$ and any sets  $M_x,M_y\in\M_k$ with $x\in M_x$ and $y\in M_y$. Since the canonical order $\preceq_k$ of $\M_k$ is linear, either $M_x\preceq_k M_y$ or $M_y\preceq_k M_x$. 
Applying Claim~\ref{cl:prec-I}, we conclude that $x\preceq y$ or $y\preceq x$.
 %First we consider the case of $M_x\preceq_k M_y$. In this case we will show that $x\preceq y$. For this take any number $n\ge m$ with $\frac6{2^n}<d_X(x,y)$ and any sets $M_x',M_y'\in\M_n$ with $x\in M_x'$ and $y\in M_y'$. The choice of $k$ ensures that $k\le n$. The inequality $x\preceq y$ will follow as soon as we show that $M'_x\preceq_n M_y'$. To derive a contradiction, assume that $M'_x\not\preceq_n M_y'$ and hence $M'_y\prec_n M_x'$ by the linearity of the order $\preceq_n$ on $\M_n$. Let $M_x''=\mu_k^n(M_x')\in\M_k$ and $M''_y=\mu_k^n(M'_y)\in\M_k$. The order-preserving property of the map $\mu^n_k:\M_n\to\M_k$ implies that $M_y''=\mu^n_k(M'_y)\preceq_k\mu_k^n(M'_x)=M''_x$. The condition (6) of Claim~\ref{cl:M} ensures that $x\in M_x'\subseteq \mu^n_k(M_x')=M_x''$ and $y\in M_y'\subseteq\mu^n_k(M_y')=M_y''$. It follows from $\diam( M_x'')+\diam(M_y'')\le2\,\mesh(\M_k)\le2\,\mesh(\C_k)\le\frac6{2^k}<d(x,y)$ that $M_x''\cap M_y''=\emptyset$ and hence $M_y''\prec_k M_x''$ and moreover, the order interval $\{M\in\M_k:M_y''\prec_k M\prec_k M_x''\}$ is not empty. This fact and $M_y\cap M_y''\ne\emptyset\ne M_x\cap M_x''$ imply $M_y\preceq_k M_x$. Taking into account that $M_x\preceq_k M_y$, we conclude that $M_x=M_y$ and $d(x,y)\le\diam (M_x)\le\mesh(\C_k)\le\frac3{2^k}<\frac6{2^k}<d(x,y)$, which is a desired contradiction showing that $M'_x\preceq_n M_y'$ and hence $x\preceq y$.
%
%By analogy we can show that $M_y\preceq_k M_x$ implies $y\preceq x$. Therefore, $x\preceq y$ or $y\preceq x$, which means that the order $\preceq$ on $I$ is linear. 
\end{proof}

The definition of the linear order $\preceq$ on $I$ implies that $a=\min (I,\preceq)$ and $b=\max (I,\preceq)$.
%Given two elements $x,y\in I$ we shall write $x\prec y$ iff $x\preceq y$ and $x\ne y$.

\begin{claim}\label{cl:intersection} For every $n\ge m$ and $M\in\M_n$ the set $\check M\defeq M\cap I\setminus\bigcup(\M_n\setminus\{M\})$ is infinite.
\end{claim}

\begin{proof} By Lemma~\ref{l:connector}, there exists an $(\{a\},\{b\})$-chain $(M_1,\dots,M_{l})$ such that $\M_n=\{M_1,\dots,M_{l}\}$. Given any $M\in\M_n$, find a unique $k\in\{1,\dots,l\}$ such that $M=M_k$. Let $M_0=\{a\}$ and $M_{l+1}=\{b\}$. By the definition of an $(\{a\},\{b\})$-chain, the sets $M_{<k}\defeq \bigcup_{i=0}^{k-1}M_i$ and $M_{>k}\defeq\bigcup_{i=k+1}^{l+1}M_i$ are disjoint and hence $I_{<k}\defeq I\cap M_{<k}$ and $I_{>k}\defeq I\cap M_{>k}$ are disjoint nonempty closed subsets of the connected space $I$. Assuming that $\check M=(I\cap M_k)\setminus(M_{<k}\cup M_{>k})$ is finite, we conclude that the set $I\setminus I_{<k}=(\check M\setminus I_{<k})\cup I_{>k}$ is closed in $I$ and hence $I=I_{<k}\cup (I\setminus I_{<k})$ is the union of two disjoint nonempty closed subsets, which contradicts the connectedness of $I$. This contradiction shows that the set $\check M$ is infinite.
\end{proof}

Using Claim~\ref{cl:intersection}, construct inductively a sequence $(c_n)_{n=m}^\infty$ of injective functions $c_n:\M_n\to I$ such that 
\begin{itemize}
\item $c_n(M)\in (I\cap M)\setminus(\{a,b\}\cup\bigcup(\M_n\setminus\{M\}))$ for every $M\in\M_n$, and
\item  $c_n[\M_n]\cap c_k[\M_k]=\emptyset$ for all $k$ with $m\le k<n$.
\end{itemize}

\begin{claim}\label{cl:mono} For any number $n\ge m$ and sets $M'\prec_n M''$ in $\M_n$ the following properties hold:
\begin{enumerate}
\item $c_n(M')\prec c_n(M'')$;
\item $\{x\in I:c_n(M')\preceq x\preceq c_n(M'')\}\subseteq \bigcup\{M\in\M_n:M'\preceq_n M\preceq_n M''\}$;
\item $\{x\in I:c_n(M')\preceq x\}\subseteq\bigcup\{M\in\M_n:M'\preceq_n M\}$;
\item $\{x\in I:x\preceq c_n(M'')\}\subseteq\bigcup\{M\in\M_n:M\preceq_n M''\}$.
\end{enumerate}
\end{claim}

\begin{proof} 1. The injectivity of the map $c_n:M_n\to I$ guarantees that $c_n(M')\ne c_n(M'')$. Assuming that $c_n(M')\not\prec c_n(M'')$, we obtain that $c_n(M'')\prec c_n(M')$ as the order $\preceq$ is linear. Choose any number $k\ge n$ such that $\frac6{2^k}<d_X(c_n(M'),c_n(M''))$ and find sets $M'_k,M_k''\in\M_k$ such that $c_n(M')\in M'_k$ and $c_n(M'')\in M_k''$. By the definition of the order $\preceq$, the inequality $c_n(M'')\prec c_n(M')$ implies that $M''_k\preceq_k M'_k$. The condition (6) of Claim~\ref{cl:M} ensures that $c_n(M')\in M'_k\subseteq\mu^k_n(M'_k)\in\M_n$ and $c_n(M'')\in M_k''\subseteq\mu^k_n(M_k'')\in\M_n$. Taking into account that the set $M'$ is a unique set in $\M_n$ containing the point $c_n(M')$, we conclude that $\mu^k_n(M'_k)=M'$. By analogy we can show that $\mu^k_n(M_k'')=M''$.  The inequality $M''_k\preceq_k M'_k$ and the order-preserving property of the map $\mu^k_n:\M_k\to\M_n$ imply that $M''=\mu^k_n(M''_k)\preceq_n \mu^k_n(M'_k)=M'$, which contradicts the strict  inequality  $M'\prec_n M''$.
\smallskip 

2.
Now take any $x\in I$ with $c_n(M')\preceq x\preceq c_n(M'')$. If $x\in\{c_n(M'),c_n(M'')\}$, then $x\in M'\cup M''\subseteq\bigcup\{M\in\M_n:M'\preceq_n M\preceq_n M''\}$ and we are done. So, we assume that $c_n(M')\prec x\prec c_n(M'')$. Choose any number $k\ge n$ such that $\frac6{2^k}<\min\{d_X(c_n(M'),x),d_X(x,c_n(M''))\}$ and then choose sets $M'_k,M_k,M_k''\in\M_k$ with $c_n(M')\in M'_k$, $x\in M_k$ and $c_n(M'')\in M''_k$. By the definition of the order $\preceq$, the inequalities $c_n(M')\prec x\prec c_n(M'')$ imply $M'_k\preceq_k M_k\preceq_k M''_k$. Repeating the argument from the preceding paragraph, we can show that $M'=\mu^k_n(M'_k)\preceq_n \mu^k_n(M_k)\preceq_n\mu^k_n(M_k'')=M''$. The condition (6) of Claim~\ref{cl:M} ensures that $$\textstyle x\in M_k\subseteq\mu^k_n(M_k)\subseteq\bigcup\{M\in\M_n:M'\preceq_n M\preceq_n M''\}.$$  
\smallskip

3,4. The last two  properties of Claim~\ref{cl:mono} can be proved by analogy with the property (2).
\end{proof}

\begin{claim}\label{cl:inter} For any points $x\prec y$ in $I$  and any number $n\ge m$ with $\frac{6}{2^n}<d_X(x,y)$ there exists $M\in\M_n$ such that $x\prec c_n(M)\prec y$.
\end{claim}

\begin{proof} Let $M_x$ be the maximal element of the linearly ordered set $(\M_n,\preceq_n)$ such that $x\in M_x$ and $M_y$ be the minimal element of $(\M_n,\preceq_n)$ such that $y\in M_y$.  The inequality $x\prec y$ implies $M_x\preceq_n M_y$. It follows from $\diam(M_x)+\diam(M_y)\le 2{\cdot}\mesh(\M_n)\le 2{\cdot}\mesh(\C_n)\le\frac6{2^n}<d_X(x,y)$ that $M_x\cap M_y=\emptyset$ and hence there exists a set $M\in\M_n$ such that $M_x\prec_n M\prec_n M_y$. The choice of the point $c_n(M)$ ensures that 
$$\textstyle{c_n(M)\in I\cap M\setminus \bigcup(\M_n\setminus\{M\})\subseteq I\cap M\setminus (M_x\cup M_y)}$$
and hence $x\ne c_n(M)\ne y$. 

Since the sets $M_x,M_y$ are closed and $c_n(M)\notin M_x\cup M_y$, there exists $k\ge n$ such that
\begin{itemize}
\item  $\frac6{2^k}<\min\{d_X(x,c_n(M)),d_X(c_n(M),y)\}$, and
\item any point $z\in X$ with $d_X(z,c_n(M))<\frac3{2^k}$ does not belong to $M_x\cup M_y$. 
\end{itemize}
Take any sets $M_x',M_y',M'\in\M_k$ such that $x\in M'_x$, $y\in M_y'$ and $c_n(M)\in M'$. Let $M_x''=\mu^k_n(M_x')$, $M_y''=\mu^k_n(M_y')$ and $M''=\mu^k_n(M')$. The condition (6) of Claim~\ref{cl:M} ensures that $x\in M_x'\subseteq M_x''$, $y\in M_y'\subseteq M_y''$ and $c_n(M)\in M'\subseteq M''$. Taking into account that $c_n(M)\in M''\cap M\setminus\bigcup(\M_n\setminus\{M\})$, we conclude that $M''=M$.  The choice of the sets $M_x$ and $M_y$ implies that $M_x''\preceq_n M_x\prec_n M=M''\prec_n M_y\preceq_n  M_y''$.  The strict inequalities $M_x''\prec_n M''\prec_n M_y''$ and the order-preserving property of the map $\mu_n^k:\M_k\to\M_n$ imply that $M_x'\prec_k M'\prec_k M_y'$ and hence $x\prec c_n(M)\prec y$, by Claim~\ref{cl:prec-I}. 
\end{proof}

\begin{claim}\label{cl:upper} For any point $x\in I$  and any number $n\ge m$ with $\frac{6}{2^n}<d_X(x,b)$ there exists $M\in\M_n$ such that $x\prec c_n(M)$.
\end{claim}

\begin{proof} Let $M_x$ be the maximal element of the linearly ordered set $(\M_n,\preceq_n)$ such that $x\in M_x$ and $M_b$ be the unique element of $\M_n$ such that $b\in M_b$.  The definition of the linear $\preceq_n$ implies $M_x\preceq_n M_b$. It follows from $\diam(M_x)+\diam(M_b)\le 2{\cdot}\mesh(\M_n)\le 2{\cdot}\mesh(\C_n)\le\frac6{2^n}<d_X(x,b)$ that $M_x\cap M_y=\emptyset$ and hence there exists a set $M\in\M_n$ such that $M_x\prec_n M$. By analogy with Claim~\ref{cl:inter} we can show $x\prec c_n(M)$.
\end{proof}

Consider the countable set $L=\bigcup_{n=m}^\infty c_n[\M_n]$ and observe that $$L\subseteq \bar L\subseteq I\subseteq \textstyle\bigcup\M_m\subseteq \bigcup\C_m[F]\subseteq O[F;\tfrac3{2^m}].$$ Let $w:L\to\IR$ be the function defined by $w(c_n(M))\defeq\delta_n$ for every $n\ge m$ and $M\in\M_n$. 

Consider the increasing function $$g:L\to \IR,\quad g:x\mapsto \sum_{z\prec x}w(z).$$ It is clear that $g[L]\subseteq [0,s]$ where $$s=\sum_{x\in L}w(x)=\sum_{n=m}^\infty|\M_n|\cdot\delta_n\le \sum_{n=m}^\infty|\C_n|\cdot\delta_n=\sum_{n=m}^\infty S_X(\tfrac1{2^n})\delta_n<\infty.$$
It is clear that the function $g:L\to [0,s]$ is injective, so we can consider the inverse function $f\defeq g^{-1}:g[L]\to L\subseteq I$.

\begin{claim}\label{cl:Omega} For any number $n\ge m$ and points $x,y\in L$ with $|g(x)-g(y)|<\delta_n$ we have $$d_X(x,y)\le \tfrac6{\;2^n}.$$
\end{claim}

\begin{proof} We lose no generality assuming that $x\prec y$. Assuming that  $d_X(x,y)>\frac6{2^n}$, we can apply Claim~\ref{cl:inter} and find a set $M\in\M_n$ such that $x\prec c_n(M)\prec y$. Then $$|g(y)-g(x)|=g(y)-g(x)=\sum_{x\preceq z\prec y}w(z)\ge w(c_n(M))=\delta_n,$$ which contradicts the choice of $x,y$. This contradiction shows that $d_X(x,y)\le \frac6{2^n}$.
\end{proof}

Claim~\ref{cl:Omega} implies   
\begin{claim}\label{cl:Omega1} The function $f:g[L]\to L$ is uniformly continuous and
$$\forall n\ge m\;\;\forall u,v\in g[L]\;\big(|u-v|<\delta_n\Rightarrow d_X(f(u),f(v))\le\tfrac6{2^n}\big).
$$
\end{claim}

Let $\bar f:\overline{g[L]}\to\overline L$ be the unique continuous function extending the uniformly continuous function $f$.  Claim~\ref{cl:Omega1}  implies

\begin{claim}\label{cl:Omega2} For any number $n\ge m$ and real numbers $u,v\in \overline{g[L]}$ with $|u-v|<\delta_n$, we have $d_X(\bar f(u),\bar f(v))\le\tfrac6{2^n}$.
\end{claim}

\begin{claim} $\{0,s\}\subseteq \overline{g[L]}$.
\end{claim}

\begin{proof} Given any $\e>0$, find a nonempty finite set $F\subseteq L$ such that $\sum_{x\in L\setminus F}w(x)<\e$. Then $g(\min F)\le \sum_{x\in L\setminus F}w(x)<\e$ and hence $0\in \overline{g[L]}$. By Claim~\ref{cl:upper}, there exists $y\in L$ such that $\max F\prec y$. Then $s-\e<\sum_{x\in F}w(x)\le \sum_{x\prec y}w(x)=g(y)\le s$ and hence $s\in\overline{g[L]}$.
\end{proof} 

\begin{claim}\label{cl:ab} $\bar f(0)=a$ and $\bar f(s)=b$.
\end{claim}

\begin{proof} To derive a contradiction, assume that $\bar f(0)\ne a$. Since $0\in \overline{g[L]}$, there exists a   sequence $\{a_k\}_{k\in\w}\subseteq L$ such that the sequence  $(g(a_k))_{k\in\w}$ is strictly decreasing and tend to zero. The  continuity of $\bar f$ ensures that $\bar f(0)=\lim_{k\to\infty}a_k$. Assuming that $\bar f(0)\ne a$, we can find a number $n\in\w$ such that $\frac4{2^n}<d_X(a,\bar f(0))$. By Lemma~\ref{l:connector}, there exists an $(\{a\},\{b\})$-chain $(M_1,\dots,M_l)$ such that $\M_n=\{M_1,\dots,M_l\}$. Since $\bar f(0)=\lim_{k\to\infty}a_k$ and $0=\lim_{k\to\infty}g(a_k)$, there exists $k\ge n$ such that $d_X(a_k,\bar f(0))<\frac1{2^n}$ and $g(a_k)<g(c_n(M_1))$. Since $g:L\to g[L]$ is an order isomorphism,  $a\prec a_k\prec c_n(M_1)$. Claim~\ref{cl:mono}(4) ensures that $a_k\in M_1$ and hence $d_X(a,a_k)\le\diam(M_1)\le\mesh(\M_n)\le\mesh(\C_n)\le\tfrac3{2^n}$ and $$d_X(a,\bar f(0))\le d_X(a,a_k)+d_X(a_k,\bar f(0))\le \tfrac3{2^n}+\tfrac1{2^n}=\tfrac4{2^n}<d_X(a,\bar f(0)).$$
This contradiction shows that $\bar f(0)=a$. 

By analogy we can prove that $\bar f(s)=b$.
\end{proof} 

\begin{claim}\label{cl:constant} For any real numbers $u<v$ with $[u,v]\cap\overline{g[L]}=\{u,v\}$ we have $\bar f(u)=\bar f(v)$.
\end{claim}

\begin{proof} Consider the sets $A=\{x\in L:g(x)\le u\}$ and $B=\{x\in L:g(x)\ge v\}$. The order-preserving property of the bijective function $g:L\to g[L]$ and the equality $[u,v]\cap\overline{g[L]}=\{u,v\}$ ensure that the sets $A,B$ are not empty, $A\cup B=L$, and $x\prec y$ for all $x\in A$ and $y\in B$. 

Consider four cases.
\smallskip

1. The set $A$ has no maximal element and the set $B$ has no minimal element. Since 
$$
\sum_{z\in L}w(z)=\sum_{n=m}^\infty\sum_{M\in\M_n}w(c_n(M))=\sum_{n=m}^\infty|\M_n|{\cdot}\delta_n\le 
\sum_{n=m}^\infty|\C_n|{\cdot}\delta_n=\sum_{n=m}^\infty S_X(\tfrac1{2^{n}}){\cdot}\delta_n<\infty,
$$ there exists a number $l>m$ such that $\sum_{k=l}^\infty\sum_{M\in\M_k} w(c_k(M))<v-u$. Since the set $A$ has no maximal element and the set $B$ has no minimal element, there are points $x\in A$ and $y\in B$ such that the order interval $[x,y]=\{z\in L:x\preceq z\preceq y]$ is disjoint with the finite set $\bigcup_{i=m}^{l-1}c_i[\M_i]$. Then $$v-u\le g(y)-g(x)=\sum_{x\preceq z\prec y}w(z)\le\sum_{k=l}^\infty\sum_{M\in\M_k}w(c_k(M))<v-u,$$
which shows that the case 1 is impossible.
\smallskip

2. The set $A$  has the largest element $x\in A$ and the set $B$ has the smallest element $y\in B$. Then the order interval $\{z\in L:x\prec z\prec y\}$ is empty, which contradicts Claim~\ref{cl:inter}. So, the case 2 is impossible, too.
\smallskip

3. The set $A$ has no largest element but the set $B$ has the smallest element $y\in B$. Then $g(y)=\sum_{z\prec y}w(z)=\sum_{z\in A}w(z)=\sup_{x\in A}\sum_{z\prec x}w(z)=\sup_{x\in A}g(x)\le u$ and hence $y\in A=L\setminus B$, which contradicts the choice of $y$.
\smallskip

4. The set $A$ has the largest element $x\in A$ and the set $B$ has no smallest element $B$.  In this case we can choose a sequence $\{b_k\}_{k\in\w}\subseteq B$ such that $b_{k+1}\prec b_k$ for every $k\in\w$ and for any $y\in B$ there exists $k\in\w$ such that $b_k\prec y$.
Taking into account that $x$ is the largest element of $A$, we conclude that $g(x)=\sum_{z\prec x}w(z)=u$ and $g(x)+w(x)=v=\lim_{k\to\infty}g(b_k)$. Then $\bar f(u)=f(u)=x$ and $\bar f(v)=\lim _{k\to\infty}b_k\in I$. We claim that $x=\bar f(v)$. Assuming that $x\ne\bar f(v)$, we can find a number $n\ge m$ such that $\frac7{2^n}<\min\{d_X(x,\bar f(v)),d_X(a,x),d_X(x,b)\}$. By Lemma~\ref{l:connector}, there exists an $(\{a\},\{b\})$-chain $(M_1,\dots,M_l)$ such that $\M_n=\{M_1,\dots,M_l\}$. Let $M_0=\{a\}$, $M_{l+1}=\{b\}$, $c_n(M_0)=a$ and $c_n(M_{l+1})=b$. Claim~\ref{cl:mono}(1) implies that $$a=c_n(M_0)\prec_n c_n(M_1)\prec_n\cdots\prec_n c_n(M_l)\prec_n c_n(M_{l+1})=b$$ and hence there exists a number $i\in\{0,\dots,l\}$ such that $c_n(M_i)\preceq_n x\prec c_n(M_{i+1})$. We claim that $i<l$. Indeed, in the opposite case, Claim~\ref{cl:mono}(3) ensures that $x\in M_l$ and hence $d_X(x,b)\le\diam(M_l)\le\mesh(\M_n)\le\mesh(\C_n)\le \frac3{2^n}<d_X(x,b)$, which is a contradiction showing that $i<l$ and hence $c_n(M_{i+1})\in B$. By the choice of the sequence $(b_k)_{k\in\w}$ with $\lim_{k\to\infty}b_k=\bar f(v)$, there exists a number $k\in\w$ such that $d_X(b_k,\bar f(v))<\frac1{2^n}$ and $b_k\prec c_n(M_{i+1})$. It follows from $x\in A$ and $b_k\in B$ that 
$$c_n(M_i)\preceq_n x\prec_n b_k\prec_n c_n(M_{i+1}).$$
By Claim~\ref{cl:mono}(2), $\{x,b_k\}\subseteq M_i\cup M_{i+1}$ and hence
$$d_X(x,b_k)\le\diam(M_i)+\diam(M_{i+1})\le2\mesh(\M_n)\le\tfrac6{2^n}$$ and finally
$$d_X(x,\bar f(v))\le d_X(x,b_k)+d_X(b_k,\bar f(v))\le \tfrac6{2^n}+\tfrac1{2^n}<d_X(x,\bar f(v)).$$
This contradiction completes the proof of the equality $\bar f(u)=x=\bar f(v)$.
\end{proof} 

By Claim~\ref{cl:constant}, the function $\bar f:\overline{g[L]}\to\bar L$ can be extended to a function $p:[0,s]\to\bar L$ such that $p$ is constant on each interval $[u,v]$ such that $[u,v]\cap\overline{g[L]}=\{u,v\}$. By Claim~\ref{cl:ab}, $p(0)=\bar f(0)=a$ and $p(s)=\bar f(s)=b$. Claim~\ref{cl:Omega2} implies that for any number $n\ge m$ and any real numbers $u,v\in[0,s]$ with $|u-v|<\delta_n$ we have $d_X(p(u),p(v))\le\tfrac6{2^n}$.
\end{proof}

The following corollary of Lemma~\ref{l:path} allows us to link points of metric Peano continua by paths with controlled continuity modulus.

\begin{corollary} Let $X$ be a compact connected metric space of diameter $\le 1$ and $\Omega:\IR_+\to\IR_+$ be an increasing function such that $(0,1]\subseteq\Omega[\IR_+] $ and $$\textstyle s\defeq \sum_{n=3}^\infty S_X(\tfrac1{2^n})\Omega^{-1}(\min\{1,\tfrac{12}{\;2^n}\})<\infty.$$
Then for any points $a,b\in X$ there exists a function $p:[0,s]\to X$ such that $p(0)=a$, $p(s)=b$ and $\w_p\le\Omega$.
\end{corollary}

\begin{proof} For every $n\ge 3$ consider the number $\delta_n\defeq\Omega^{-1}(\min\{1,\frac{12}{\;2^n}\})$ which is well-defined because $(0,1]\subseteq\Omega[\IR_+]$. By Lemma~\ref{l:path}, for every $a,b\in X$ there exists a function $p:[0,s]\to X$ such that $p(0)=a$, $p(s)=b$ and for every $n\ge 3$ and $x,y\in[0,s]$ with $|x-y|<\delta_n$ we have $d_X(p(x),p(y))\le\frac6{2^n}$. 

We claim that $d_X(p(x),p(y))\le\Omega(|x-y|)$ for any points $x,y\in [0,s]$. If $|x-y|\ge\Omega^{-1}(1)$, then $d_X(p(x),p(y))\le\diam(X)\le 1=\Omega(\Omega^{-1}(1))\le\Omega(|x-y|)$ and we are done. So, assume that $|x-y|<\Omega^{-1}(1)=\delta_3$ and then $\delta_{n+1}\le |x-y|<\delta_n$ for a unique number $n\ge 3$. The choice of the function $p$ guarantees that 
$$d_X(p(x),p(y))\le\tfrac{6}{2^n}=\Omega(\Omega^{-1}(\tfrac6{2^n}))=\Omega(\Omega^{-1}(\tfrac{12}{2^{n+1}}))=\Omega(\delta_{n+1})\le\Omega(|x-y|).$$
\end{proof}

\section{A controlled Aleksandroff--Urysohn Theorem}\label{s:HM}

By the classical Aleksandrov--Urysohn Theorem \cite[4.18]{K}, every metrizable compact space is a continuous image of the Cantor set. In the following lemma we prove a modification of the Aleksandrov--Urysohn Theorem for maps onto metric Peano continua giving some extra-information on metric connectedness properties of such maps.

\begin{lemma}\label{l:gap} Let $(\e_n)_{n\in\w}$ be a decreasing sequence of positive real numbers and $X$ be a compact connected metric space of nonzero diameter $\le 1$. If $s\defeq\sum_{n=1}^\infty S_X(\tfrac1{2^{n}})\e_n<\infty$, then there exists a closed subset $D\subseteq[0,s]$ with $\{0,s\}\subseteq D$ and a continuous surjective map $\bar f:D\to X$ such that 
\begin{enumerate}
\item for any $n\in\IN$ and real numbers $u<v$ in $D$ with $|u-v|<\e_n$ we have $d_X(\bar f(u),\bar f(v))\le \frac4{2^n}$;
\item for any real numbers $u,v\in D$ with $[u,v]\cap D=\{u,v\}$ there exists $n\in\IN$ such that $v-u=\e_n$ and the doubleton $\{\bar f(u),\bar f(v)\}$ is contained in a connected subset $C\subseteq X$ of diameter $\diam(C)\le \tfrac{\!8}{2^n}$.
\end{enumerate}
\end{lemma} 

\begin{proof} Let $\C_0=\{X\}$ and for every $n\in\IN$ choose a cover $\C_n$ of $X$ by closed connected subsets such that $\mesh(\C_n)\le \frac1{2^{n}}$ and $|\C_n|=S_X(\frac1{2^{n}})$. The connectedness of $X$ and the minimality of $S_X(2^{-n})$ ensure that each set $C\in\C_n$ contains more than one point and by the connectedness, has cardinality of continuum. 

%Let $\le_0$ be the unique linear order on the singleton $\C_0=\{X\}$. For every $n\in\w$ we can define a linear order $\preceq_n$ on the set $\C_n$ such that the function 
%$$\mu^{n+1}_n:\C_{n+1}\to \C_n,\quad \mu^{n+1}_n:C\mapsto \min\{C'\in \C_n:C\cap C'\ne\emptyset\},$$is a monotone map of the linearly ordered sets $(\C_{n+1},\preceq_{n+1})$ and $(\C_n,\preceq_n)$.

%It follows from $\mu_n^{n+1}(C)\cap  C\ne\emptyset$ that 
%$$\diam(C\cup\mu_n^{n+1}(C))\le \diam(C)+\diam(\mu_n^{n+1}(C))\le\mesh(\C_{n+1})+\mesh(\C_n)\le \tfrac1{2^{n+1}}+\tfrac1{2^n}=\tfrac3{2^n}$$for every $C\in\C_{n+1}$.

%For every $n<m$ let $\mu_n^n:\C_n\to\C_n$ be the identity map of $\C_n$ and $\mu_n^m:\C_m\to\C_n$ be the function defined by the recursive formula $\mu^m_n\defeq\mu^{m}_{m-1}\circ\mu_n^{m-1}$.
%It can be shown inductively that for any $n\le m$ the function $\mu^m_n:\C_m\to\C_n$ between the linearly ordered sets $(\C_m,\preceq_m)$ and $(\C_n,\preceq_n)$ is monotone. 

%Given any numbers $n\le m$ and set $C\in\C_m$, consider the sequence $(C_i)_{i=m}^n$ defined by $C_m=C$ and $C_{i-1}=\mu^i_{i-1}(C_i)$ for $i\in\{m,\dots,n+1\}$. It follows from $C_i\cap C_{i-1}=C_i\cap\mu_{i-1}^i(C_i)\ne\emptyset$ that $$\diam\big(\bigcup_{i=n}^mC_i\big)\le\sum_{i=n}^m\diam(C_i)\le\sum_{i=n}^m\mesh(\C_i)\le\sum_{i=n}^m\tfrac1{2^i}=\tfrac1{2^{n-1}}-\tfrac1{2^m}.$$ 

Then for every $n\in\w$ we can choose an injective function $c_n:\C_n\to X$ such that 
\begin{itemize}
\item  $c_n(C)\in C$ for every $C\in\C_n$, and
\item $c_n[\C_n]\cap c_k[\C_k]=\emptyset$ for all $k<n$.
\end{itemize}
For every $n\in\w$ consider the finite set $L_n=\bigcup_{k\le n}c_k[\C_k]$ and let $L=\bigcup_{n\in\w}L_n$. It follows from $\lim_{n\to\infty}\mesh(\C_n)=0$ that the set $L$ is dense in $X$.

For every $n\in\w$ choose a function $\cap_n:\C_{n+1}\to \C_n$ assigning to each set $C\in\C_{n+1}$ a set $C'\in\C_n[C]$. 
 Let $r^{n+1}_n:L_{n+1}\to L_n$ be the function defined by $$r_n^{n+1}(x)=\begin{cases}
x&\mbox{if $x\in L_n$};\\
c_n(\cap_n(c_{n+1}^{-1}(x)))&\mbox{if $x\in L_{n+1}\setminus L_n=c_{n+1}[\C_{n+1}]$}.
\end{cases}
$$  The definition of the map $r^{n+1}_n$ guarantees that $r^{n+1}_n[L_{n+1}\setminus L_n]\subseteq c_n[\C_n]=L_n\setminus L_{n-1}$, where $L_{-1}=\emptyset$.

For every $n<m$ let $r_n^n:L_n\to L_n$ be the identity map of $L_n$ and $r_n^m:L_m\to L_n$ be the function defined by the recursive formula  $r^m_n= r_n^{m-1}\circ r^{m}_{m-1}$. By induction it can be proved that $r^m_n{\restriction}_{L_n}$ is the identity function of $L_n$ and $r^m_n[L_m\setminus L_n]\subseteq L_n\setminus L_{n-1}$.

Let  $r_n:L\to L_n$ be a unique function such that $r_n{\restriction}_{L_m}=r^m_n$ for every $m\ge n$. It follows that $r_n[L\setminus L_n]\subseteq L_n\setminus L_{n-1}$.

\begin{claim}\label{cl:rnx} For every $n\in\w$ and $x\in L$, the set $\{x,r_n(x)\}$ is contained in a connected subset of diameter $<\tfrac{2}{2^n}$ in $X$.
\end{claim}

\begin{proof} Let $m\in\w$ be a unique number such that $x= c_m(C_m)\in c_m[\C_m]=L_m\setminus L_{m-1}$ for a unique set $C_m\in\C_m$. If $m\le n$, then $r_n(x)=x$ and $\{x,r_n(x)\}=\{x\}$ is a connected set of diameter $0<\tfrac2{2^{n}}$. So, we assume that $m>n$. For every $i<m$, let $x_i=r_i(x)$ and observe that $x_i=r^{i+1}_i(x_{i+1})=c_i(C_i)\in L_i\setminus L_{i-1}$ for a unique set $C_i\in\C_i$. The definition of the map $r^{i+1}_i$ ensures that $C_i=\cap_i(C_{i+1})$ and hence $C_i\cap C_{i+1}\ne\emptyset$. Then   the set $C\defeq \bigcup_{i=n}^mC_i$ is connected and has $$\diam(C)\le \sum_{i=n}^m\diam(C_i)\le\sum_{i=n}^m\mesh(\C_i)\le\sum_{i=n}^m\tfrac1{2^i}<\tfrac2{2^{n}}.$$ It is clear that $\{x,r_n(x)\}\subseteq C_m\cup C_n\subseteq C$.
\end{proof}

Let $\preceq_0$ be the unique linear order on the singleton $L_0=c_0[\C_0]=\{c_0(X)\}$. For every $n\in\IN$ choose inductively a linear order ${\preceq_{n}}\subseteq L_{n}\times L_{n}$ on the set $L_{n}$ such that
\begin{itemize}
\item[(i)] ${\preceq_{n-1}}\subseteq {\preceq_{n}}$;
\item[(ii)] the function $r^{n}_{n-1}:(L_{n},\preceq_n)\to (L_{n-1},\preceq_{n-1})$ is order-preserving;
\item[(iii)] for any $x\in L_n$ we have $x\preceq_n r^{n}_{n-1}(x)$.
\end{itemize} 
Then ${\preceq}\defeq\bigcup_{n\in\w}\preceq_n$ is a linear order on the set $L$ such that for every $n\le m$ the maps $r^m_n:L_m\to L_n$ and $r_n:L\to L_n$ are order-preserving, and $x\preceq r_n(x)$ for any $x\in L$ and $n\in\w$. The condition (iii) implies that $c_0(X)$ is the largest element of the set $L$. %Given two elements $x,y\in L$, we write $x\prec y$ iff $x\preceq y$ and $x\ne y$.

%It can be shown inductively that for any $n\le m$ the function $\mu^m_n:\C_m\to\C_n$ between the linearly ordered sets $(\C_m,\preceq_m)$ and $(\C_n,\preceq_n)$ is monotone. 

%let $r^{n+1}_n:\defeq\{c_n(C):C\in\C_n\}$ with a unique linear order $\le_n$ such that the function $c_n:\C_n\to c_n[\C_n]$ is an isomorphism of the linearly ordered spaces $(\C_n,\preceq_n)$ and $(L_n,\le_n)$. Then for every $n\le m$ the function $r^m_n\defeq c_n\circ\mu^m_n\circ c_m^{-1}:L_m\to L_n$ is a monotone map between the linearly ordered sets $(L_m,\le_m)$ and $(L_n,\le_n)$.

%Given any numbers $n\le m$ and set $C\in\C_m$, consider the sequence $(C_i)_{i=m}^n$ defined by $C_m=C$ and $C_{i-1}=\mu^i_{i-1}(C_i)$ for $i\in\{m,\dots,n+1\}$. It follows from $C_i\cap C_{i-1}=C_i\cap\mu_{i-1}^i(C_i)\ne\emptyset$ that $$
%\begin{aligned}
%\rho(c_m(C),&r^m_n(c_m(C)))=\rho(c_m(C),c_n(\mu^m_n(C)))=\rho(c_m(C_m),c_n(C_n))\le\\
%&\le\diam\big(\bigcup_{i=n}^mC_i\big)\le\sum_{i=n}^m\diam(C_i)\le\sum_{i=n}^m\mesh(\C_i)\le\sum_{i=n}^m\tfrac1{2^i}=\tfrac1{2^{n-1}}-\tfrac1{2^m}.
%\end{aligned}
%$$

%Consider the countable set $L\defeq\bigcup_{n\in\w}L_n$ and let $\hbar:L\to \w$ be a unique map such that $\hbar^{-1}(n)=L_n$ for all $n\in\w$. Endow the set $L$ with the linear order $\le$ such that $x\le y$ if and only if for the numbers $n=\hbar(x)$ and $m=\hbar(y)$ one of the following conditions is satisfied:
%\begin{itemize}
%%\item $n=m$ and $x\le_n y$;
%\item $n<m$ and $x\le _n r^m_n(y)$;
%\item $n>m$ and $r^n_m(x)<_m y$.
%\end{itemize}

Define the function $w:L\to \IR$ letting $w(c_n(C))\defeq\e_n$ for any $C\in\C_n$ and $n\in\w$. Consider the map $$g:L\to\IR,\quad g:x\mapsto \sum_{y\prec x}w(y),$$ and observe that $g[L]\subseteq [0,s]$ where  $$g(c_0(X))=g(\max L)=\sum_{n=1}^\infty\sum_{x\in c_n[\C_n]}w(x)=\sum_{n=1}^\infty|\C_n|\cdot\e_n=\sum_{n=1}^\infty S_X(\tfrac1{2^{n}}){\cdot}\e_n=s$$and hence $s\in g[L]$. The map $g:L\to [0,s]$ is increasing with respect to the linear order $\prec$ on $L$, so we can consider the inverse map $f\defeq g^{-1}:g[L]\to L$.

\begin{claim}\label{cl:gu} For any $n\in\w$ and $x,y\in L$ with $|g(x)-g(y)|<\e_n$ we have $d_X(x,y)< \frac4{2^n}$.
\end{claim}

\begin{proof}  Without loss of generality, we can assume that $x\prec y$ in the linearly ordered set $(L,\preceq)$. Find unique numbers $n_x,n_y\in\IN$ such that $x\in c_{n_x}[\C_{n_x}]$ and $y\in c_{n_y}[\C_{n_y}]$. 
It follows from $\e_n>|g(y)-g(x)|=\sum_{x\preceq z\prec y}w(z)\ge w(x)=\e_{n_x}$ that $n_x>n$.

Let $k\in\w$ be the largest number such that $r_k(x)= r_k(y)$. Since $r_0(x)=c_0(X)=r_0(y)$ and $r_i(x)=x\ne y=r_i(y)$ for all $i\ge\max\{n_x,n_y\}$, the number $k$ is well-defined and $k<\max\{n_x,n_y\}$. 

The maximality of $k$ guarantees that $r_{k+1}(x)\ne r_{k+1}(y)$ and hence $r_{k+1}(x)\prec  r_{k+1}(y)$. We claim that $r_{k+1}(x)\prec y$. Indeed, in the opposite case, $y\preceq r_{k+1}(y)$ and $r_{k+1}(y)\preceq r_{k+1}\circ r_{k+1}(x)=r_{k+1}(x)$, which contradicts the strict inequality $r_{k+1}(x)\prec r_{k+1}(y)$. 

The definition of the linear order $\preceq$ on the set $L$ ensures that $x\preceq r_{k+1}(x)\prec y$.   If $k<n$, then   
$$|g(y)-g(x)|=\sum_{x\preceq z\prec y}w(z)\ge w(r_{k+1}(x))=\e_{k+1}\ge\e_n,$$which contradicts the choice of the points $x,y$. This contradiction shows that $k\ge n$ and hence 
 $$d_X(x,y)\le d_X(x,r_k(x))+d_X(r_k(x),r_k(y))+d_X(r_k(y),y) <\tfrac2{2^k}+0+\tfrac{2}{2^k}=\tfrac4{2^k}\le\tfrac4{2^n},$$
according to Claim~\ref{cl:rnx}.
\end{proof}

Claim~\ref{cl:gu} implies 

\begin{claim}\label{cl:fu} For any $n\in\w$ and $u,v\in g[L]$ with $|u-v|<\e_n$ we have $d_X(f(u),f(v))< \frac4{2^n}$.
\end{claim}

By Claim~\ref{cl:fu}, the function $f:g[L]\to L$ is uniformly continuous and hence admits a continuous extension $\bar f:D\to X$, defined on the closure $D=\overline{g[L]}$ of the set $g[L]$ in the interval $[0,s]$. The compactness of $D$ and the density of $L$ in $X$ ensure that the map $\bar f:D\to X$ is surjective.

\begin{claim} $\{0,s\}\subseteq D$.
\end{claim}

\begin{proof} Given any $\e>0$, find a nonempty finite set $F\subseteq L$ such that $\sum_{x\in L\setminus F}w(x)<\e$. Then $g(\min F)\le \sum_{x\in L\setminus F}w(x)<\e$ and hence $0\in \overline{g[L]}=D$.

The inclusion $s\in D$ follows from $s=g(c_0(X))=g(\max L)\in g[L]\subseteq D$.
\end{proof}

Claim~\ref{cl:fu} implies 

\begin{claim} For any $n\in\w$ and $u,v\in D$ with $|u-v|<\e_n$ we have $d_X(\bar f(u),\bar f(v))\le \frac4{2^n}$.
\end{claim}%\footnote{FS: Napisac jako osobny Claim.}

\begin{claim} For any real numbers $u<v$ with $[u,v]\cap D=\{u,v\}$ there exists $n\in\IN$ such that $v-u=\e_n$ and the doubleton $\{\bar f(u),\bar f(v)\}$ is contained in a connected subset $C\subseteq X$ of $\diam(C)\le\tfrac8{2^n}$.
\end{claim}

\begin{proof} Consider the subsets $A=\{x\in L:g(x)\le u\}$ and $B=\{x\in L:v\le g(x)\}$  of the linearly ordered set $(L,\preceq)$  and observe that $A,B$ are nonempty, $A\cup B=L$ and $x\prec y$ for any $x\in A$ and $y\in B$.  Then the four cases are possible.
\smallskip

1. The set $A$ has the largest element $\max A$ and the set $B$ has the smallest element $\min B$ in the linearly ordered set $(L,\preceq)$. Then $g(\max A)=\sum_{x\prec \max A}w(x)\le u$ and $g(\min B)=\sum_{x\prec \min B}w(x)=\sum_{z\preceq\max A}w(z)=g(\max A)+w(\max A)\ge v$, which implies  $g(\max A)=u$, $g(\max B)=v$, and $w(\max A)=v-u$. Find unique numbers $n,m\in\w$ such that $\max A\in c_n[\C_n]$ and $\min B\in c_m[\C_m]$ and observe that $v-u=w(\max A)=\e_n$. 
We claim that $r_{n-1}(\max A)=r_{n-1}(\min B)$. To derive a contradiction, assume that  $r_{n-1}(\max A)\ne r_{n-1}(\min B)$ and hence $r_{n-1}(\max A)\prec r_{n-1}(\min B)$, by the order-preserving property of the map $r_{n-1}$. Taking into account that $\max A\prec r_{n-1}(\max A)$ and $\{x\in L:\max A\prec x\prec \min B\}=\emptyset$, we conclude that $\min B\preceq r_{n-1}(\max A)$ and hence $r_{n-1}(\min B)\preceq r_{n-1}(r_{n-1}(\max A))=r_{n-1}(\max A)$, which contradicts the strict inequality $r_{n-1}(\max A)\prec r_{n-1}(\min B)$. This contradiction shows that $r_{n-1}(\max A)=r_{n-1}(\min B)$. By Claim~\ref{cl:rnx}, there are connected sets $C_A,C_B\subseteq X$ of diameter $<\frac2{2^{n-1}}$ such that $\{\max A,r_{n-1}(\max A)\}\subseteq C_A$ and $\{\min B,r_{n-1}(\min B)\}\subseteq C_B$. Since $r_{n-1}(\max A)=r_{n-1}(\min B)\in C_A\cap C_B$, the set $C=C_A\cup C_B$ is connected and has diameter $\diam(C)\le\diam (C_A)+\diam (C_B)<\frac8{2^n}$. It is clear that $\{\bar f(u),\bar f(v)\}=\{\max A,\min B\}\subseteq C$. 
\smallskip

2. The set $A$ has the largest element $\max A$ but the set $B$ has no smallest element. In this case $\max A=\inf B$, $g(\max A)\le u$ and $v<g(x)\ge g(\max A)+w(\max A)$ for every $x\in B$, which implies that $g(\max A)=u$, $v=g(\max A)+w(\max A)$ and hence $w(\max A)=v-u$. Find a unique number $n\in\IN$ such that $\max A\in c_n[\C_n]$ and observe that $v-u=w(\max A)=\e_n$.  The definition of the linear order $\preceq$ implies that $\max A\prec r_{n-1}(\max A)$. Then for any $x\in L$ with $\max A\preceq  x\preceq r_{n-1}(\max A)$ we have $r_{n-1}(\max A)\preceq r_{n-1}(x)\preceq r_{n-1}(r_{n-1}(\max A))=r_{n-1}(\max A)$ and hence $r_{n-1}(x)=r_{n-1}(\max A)$. By Claim~\ref{cl:rnx}, there exists a connected subspace $C_x\subseteq X$ of $\diam(C_x)<\frac2{2^{n-1}}$ such that $\{x,r_{n-1}(\max A)\}=\{x,r_{n-1}(x)\}\subseteq C_x$. Then the set $C=\bigcup\{C_x:\max A\preceq x\preceq r_{n-1}(x)\}$ is connected, has diameter $\diam (C)\le \frac8{2^n}$ and contains the set $[\max A,r_{n-1}(\max A)]=\{x\in L:\max A\preceq x\preceq r_{n-1}(\max A)\}$. The definition of the map $\bar f:D\to \bar L$ ensures that $\bar f(v)\in\overline{[\max A,r_{n-1}(\max A)]}\subseteq\overline C$. Taking into account that $\max A=f(u)=\bar f(u)$, we conclude that $\{\bar f(u),\bar f(v)\}\subseteq \overline C$ and $\diam (\overline C)=\diam (C)\le\frac8{2^n}$.
\smallskip

3. The set $A$ does not have the largest element but the set $B$ has the smallest element $\min B$. Taking into account that $\sum_{x\prec a}w(x)=g(a)\le u$ for all $a\in A$, we conclude that $g(\min B)=\sum_{a\in A}w(a)\le u$ and hence $\min B\in A$, which contradicts  $\min B\in B$.
\smallskip

4. The set $A$ does not have the largest element and the set $B$ does not have the smallest element. Since $\sum_{x\in L}w(x)<\infty$, there exists $n\in\w$ such that $\sum_{x\in L\setminus L_n}w(x)<v-u$. Choose elements $x\in A$ and $y\in B$ such that the order segment $[x,y]=\{z\in L:x\preceq z\preceq y]$ is disjoint with the finite set $L_n$. Then $$v-u\le g(y)-g(x)=\sum_{x\preceq z\prec y}w(x)\le\sum_{z\in L\setminus L_n}w(x)<v-u,$$which is a contradiction completing the proof of the claim.
\end{proof}

\end{proof}

\section{A controlled Hahn-Mazurkiewicz Theorem}

In this section we combine Lemmas~\ref{l:path} and \ref{l:gap} in order to prove the following lemma that will imply the controlled Hahn-Mazurkiewicz Theorem~\ref{t:main}.

\begin{lemma}\label{l:sur} Let $(\delta_n)_{n\in\w}$ be a decreasing sequence of positive real numbers and $X$ be a compact connected metric space of diameter $\le 1$. If $s\defeq \sum_{n=1}^\infty S_X(\tfrac1{2^{n}})\sum_{m=n}^\infty S_X(\tfrac1{2^m})\delta_m<\infty$, then there exists a continuous surjective map $f:[0,s]\to X$ such that for any $n\in\w$ and real numbers $u,v\in [0,s]$ with $|u-v|<\delta_n$ we have $d_X(f(u),f(v))\le \frac{2^5}{\,2^n}$.
\end{lemma} 

\begin{proof} For every $n\in\w$ let $\e_n\defeq\sum_{m=n}^\infty S_X(\frac1{2^m})\delta_m$ and observe that $\delta_n<\e_n$. It is clear that the sequence $(\e_n)_{n\in\w}$ is decreasing and $s=\sum_{n=1}S_X(\frac1{2^n})\e_n$. By Lemma~\ref{l:gap}, there exists a surjective function $\varphi:D\to X$ defined on a closed subset $D\subseteq[0,s]$ such that
\begin{enumerate}
\item[(i)] $\{0,s\}\subseteq D$;
\item[(ii)] for any $n\in\IN$ and real numbers $u,v\in D$ with $|u-v|<\e_n$ we have $d_X(\varphi(u),\varphi(v))\le \frac4{2^n}$;
\item[(iii)] for any real numbers $u<v$ in $D$ with $[u,v]\cap D=\{u,v\}$ there exists $n\in\IN$ such that $v-u=\e_n$ and the doubleton $\{\varphi (u),\varphi (v)\}$ is contained in a connected subset $F\subseteq X$ of diameter $\diam(F)\le\tfrac8{2^n}$.
\end{enumerate}
Let $W$ be the set of all pairs $(u,v)\in D\times D$ such that $u<v$ and $[u,v]\cap D=\{u,v\}$. By (ii), for every pair $(u,v)\in W$ there exists a unique number $n(u,v)\in\IN$ such that $v-u=\e_{n(u,v)}$ and the doubleton $\{\varphi(u),\varphi(v)\}$ is contained in a connected set $F_{u,v}\subseteq X$ of diameter $\le\frac{8}{2^{n(u,v)}}$.

By Lemma~\ref{l:path}, there exists a continuous function $p_{u,v}:[u,v]\to O[F_{u,v},\tfrac3{2^{n(u,v)}}]$ such that $p_{u,v}(u)=\varphi(u)$, $p_{u,v}(v)=\varphi(v)$ and for every $n\ge n(u,v)$ and 
$x,y\in [u,v]$ with $|x-y|<\delta_n$ we have $d_X(p_{u,v}(x),p_{u,v}(y))\le \frac6{2^n}$.

Extend the function $\varphi:D\to X$ to a unique function $f:[0,s]\to X$ such that $f{\restriction}_{[u,v]}=p_{u,v}$ for all $(u,v)\in W$. The surjectivity of the function $\varphi$ implies the surjectivity of the function $f$.

\begin{claim} For every $n\in\w$ and $x,y\in [0,s]$ with $|x-y|<\delta_n$ we have $d_X(f(x),f(y))\le\tfrac{2^5}{\,2^n}$.
\end{claim}

\begin{proof} Separately we consider five cases.
\smallskip

1. If $x,y\in D$, then the inequality $d_X(f(x),f(y))\le \frac4{2^n}<\frac{2^5}{2^n}$ follows from (ii) and the inequality $\delta_n\le\e_n$.
\smallskip

2. If $[x,y]\subseteq [u,v]$ for some pair $(u,v)\in W$, then $|x-y|\le|u-v|=\e_{n(u,v)}$. If $n>n_{u,v}$, then the choice of the map $p_{u,v}$ ensures that $$d_X(f(x),f(y))=d(p_{u,v}(x),p_{u,v}(y))\le \tfrac6{2^n}<\tfrac{2^5}{2^n}.$$ If $n\le n_{u,v}$, then $$d_X(f(x),f(y))\le\diam\, p_{u,v}\big[[u,v]\big]\le\diam (O[F_{u,v};\tfrac3{2^{n(u,v)}}])\le\diam(F_{u,v})+\tfrac{6}{2^{n(u,v)}}\le%\tfrac{8+6}{2^{n(u,v)}} \le 
\tfrac{14}{2^n}<\tfrac{2^5}{2^n}.$$
\smallskip

3. If $x\in D$, $y\notin D$ and $[x,y]\not\subseteq [u,v]$ for all $(u,v)\in W$, then we can find a unique pair $(u,v)\in W$ such that $x<u<y<v$. Since $\max\{|x-u|,|u-y|\}<|x-y|<\delta_n$, two preceding cases imply 
$$d_X(f(x),f(y))\le d_X(f(x),f(u))+d_X(f(u),f(y))\le \tfrac4{2^n}+\tfrac{14}{2^n}=\tfrac{18}{2^n}<\tfrac{2^5}{2^n}.$$
\smallskip

4. If $x\notin D$, $y\in D$ and $[x,y]\not\subseteq[u,v]$ for all $(u,v)\in W$, then there exists a unique pair $(u,v)\in W$ such that $u<x<v<y$ and the first two cases imply
$$d_X(f(x),f(y))\le d_X(f(x),f(v))+d_X(f(v),f(y))\le \tfrac{14}{2^n}+\tfrac4{2^n}=\tfrac{18}{2^n}<\tfrac{2^5}{2^n}.$$
\smallskip

5. If $x,y\notin D$ and $[x,y]\not\subseteq [u,v]$ for all $(u,v)\in W$, then we can find unique pairs $(u_x,v_x),(u_y,v_y)\in W$ such that $u_x<x<v_x\le u_y<y<v_y$. Since $\max\{|x-v_x|,|v_x-u_y|,|u_y-y|\}<|x-y|<\delta_n$, the first two cases imply
$$d_X(f(x),f(y))\le d_X(f(x),f(v_x))+d_X(f(v_x),f(u_y))+d_X(f(u_y),f(y))\le \tfrac{14}{2^n}+\tfrac{4}{2^n}+\tfrac{14}{2^n}=\tfrac{2^5}{2^n}.$$
\end{proof}
\end{proof}

\section{Proof of Theorem~\ref{t:main}}\label{s:main}

Let $X$ be a nonempty compact connected metric space of diameter $\le 1$ and $\Omega:\IR_+\to\IR_+$ be an increasing function such that $(0,1]\subseteq\Omega[\IR_+] $ and $$s\defeq \sum_{n=1}^\infty S_X(\tfrac1{2^n})\sum_{m=n}^\infty S_X(\tfrac1{2^m})\,\Omega^{-1}(\min\{1,\tfrac{2^6}{2^m}\})<\infty.$$

We shall construct a surjective function $f:[0,s]\to X$ with continuity modulus $\w_f\le\Omega$.

 For every $n\in\w$ consider the number $\delta_n\defeq\Omega^{-1}(\min\{1,\frac{2^6}{2^n}\})$ which is well-defined because $(0,1]\subseteq\Omega[\IR_+]$. Since the function $\Omega$ is increasing, the sequence $(\delta_n)_{n\in\w}$ is decreasing (but not strictly). By Lemma~\ref{l:sur}, there exists a surjective function $f:[0,s]\to X$ such that for every $n\in\IN$ and $x,y\in[0,s]$ with $|x-y|<\delta_n$ we have $d_X(f(x),f(y))\le\frac{2^5}{2^n}$. 

We claim that $d_X(f(x),f(y))\le\Omega(|x-y|)$ for any points $x,y\in [0,s]$. If $|x-y|\ge\Omega^{-1}(1)$, then $d_X(f(x),f(y))\le\diam(X)\le 1=\Omega(\Omega^{-1}(1))\le\Omega(|x-y|)$ and we are done. So, assume that $|x-y|<\Omega^{-1}(1)=\delta_6$ and then $\delta_{n+1}\le |x-y|<\delta_n$ for a unique number $n\ge 6$. The choice of the function $f$ guarantees that 
$$d_X(f(x),f(y))\le\tfrac{2^5}{2^n}=\Omega\big(\Omega^{-1}(\tfrac{2^5}{2^n})\big)=\Omega\big(\Omega^{-1}(\tfrac{2^6}{2^{n+1}})\big)=\Omega(\delta_{n+1})\le\Omega(|x-y|).$$


\begin{thebibliography}{BNS}

\bibitem{BNS} T.~Banakh, M.~Nowak, F.~Strobin, {\em Detecting Euclidean fractals among finite-dimenional Peano continua}, preprint.

\bibitem{BT} T.~Banakh, M.~Tuncali, {\em Controlled Hahn-Mazurkiewicz theorem and some new dimension functions of Peano continua}, Topology Appl. {\bf 154}:7 (2007), 1286--1297. 

\bibitem{H} H. Hahn, {\em  Mengentheoretische characterisierung der stetigen kurven}, Sitzungsberichte Akad. der Wissenschaften {\bf 123} (1914), 24--33.

\bibitem{E} R.~Engelking, {\em General Topology}, Heldermann Verlag, Berlin, 1989.

\bibitem{K} A.~Kechris, {\em Classical Descriptive Set Theory}, Springer, 1995. 

\bibitem{M} S. Mazurkiewicz, {\em Sur les lignes de Jordan}, Fund. Math. {\bf 1} (1920), 166--209.

\bibitem{N} S.~Nadler, {\em Continuum Theory}, Marcel Dekker, Inc., New York, 1992.  

\bibitem{PS} L.~Pontrjagin, L.~Schnirelmann, {\em Sur une propri\'et\'e m\'etrique de la dimension}, Ann. Math. (2) {\bf 33}:1 (1932) 156--162.

\bibitem{S} W.~Sierpi\'nski, {\em Sur une condition pour qu'un continu soit une courbe jordanienne}, Fund. Math. {\bf 1} (1920), 44--60. 


\end{thebibliography}
\end{document}